\documentclass[review,hidelinks,onefignum,onetabnum]{siamart250211}

\usepackage{cleveref}
\usepackage{todonotes}
\usepackage{lipsum}
\usepackage{amsfonts}
\usepackage{graphicx}
\usepackage{epstopdf}
\usepackage{algorithmic}
\ifpdf
  \DeclareGraphicsExtensions{.eps,.pdf,.png,.jpg}
\else
  \DeclareGraphicsExtensions{.eps}
\fi

\usepackage{siunitx}
\newsiamremark{remark}{Remark}
\newsiamremark{hypothesis}{Hypothesis}
\crefname{hypothesis}{Hypothesis}{Hypotheses}
\newsiamremark{example}{Example}
\crefname{example}{Example}{Examples}

\newsiamthm{claim}{Claim}
\newsiamremark{fact}{Fact}
\crefname{fact}{Fact}{Facts}

\headers{KRSysId with GvR}{Z. SHEN, J. ZHANG, M. S. ANDERSEN, and T. CHEN}

\title{Numerically Efficient and Stable Algorithms for Kernel-Based Regularized System Identification using Givens-Vector Representation%
\thanks{\funding{This work was funded by 
NSFC under contract No.\ 62273287, Shenzhen Science and Technology Innovation Commission under contract No.\ JCYJ20220530143418040, 
the Novo Nordisk Foundation under contract No.\ NNF20OC0061894, the Science and Technology Ph.D.\ Research Startup Project under contract No.\ SZIIT2025KJ060,
and the Postgraduate studentships of The Chinese University of Hong Kong.}}}

\author{Zhuohua Shen\thanks{Department of Statistics and Data Science, The Chinese University of Hong Kong, Hong Kong, China (\email{zhuohuashen@link.cuhk.edu.hk}).}
\and Junpeng Zhang\thanks{School of Artificial Intelligence, Shenzhen University of Information Technology, Shenzhen, China (\email{junpengzhang@link.cuhk.edu.cn}).}
\and Martin S.\ Andersen\thanks{Department of Applied Mathematics and Computer Science, Technical University of Denmark, Lyngby, Denmark (\email{mskan@dtu.dk})}
\and Tianshi Chen\thanks{School of
Data Science and Shenzhen Research Institute of Big Data, The Chinese University of Hong Kong, Shenzhen, Shenzhen, China (\email{tschen@cuhk.edu.cn}).}}

\usepackage{amsopn}
\DeclareMathOperator{\diag}{diag}

\def\N{\mathbb{N}}
\def\R{\mathbb{R}}

\def\b{\boldsymbol b}
\def\c{\boldsymbol c}
\def\d{\boldsymbol d}
\def\e{\boldsymbol e}
\def\f{\boldsymbol f}
\def\g{\boldsymbol g}
\def\p{\boldsymbol p}
\def\q{\boldsymbol q}
\def\r{\boldsymbol r}
\def\s{\boldsymbol s}

\def\u{\boldsymbol u}

\def\w{\boldsymbol w}
\def\x{\boldsymbol x}
\def\y{\boldsymbol y}
\def\z{\boldsymbol z}

\def\boldalpha{\boldsymbol \alpha}

\def\boldchi{\boldsymbol \chi}
\def\boldeta{\boldsymbol \eta}
\def\boldmu{\boldsymbol \mu}
\def\boldnu{\boldsymbol \nu}

\def\0{\boldsymbol{0}}
\def\1{\boldsymbol{1}}

\def\calG{\mathcal{G}}
\def\calH{\mathcal{H}}

\def\calK{\mathcal{K}}

\def\calO{\mathcal{O}}

\def\calS{\mathcal{S}}

\def\calX{\mathcal{X}}

\def\hv{\hat{\nu}}
\def\hbv{\hat{\boldnu}}
\def\tbv{\tilde{\boldnu}}

\def\abs#1{\lvert #1 \rvert}
\def\inp#1#2{\left\langle #1,#2 \right\rangle}

\def\tril{\mathrm{tril}}
\def\triu{\mathrm{triu}}
\def\tr{\mathrm{tr}}
\def\diag{\mathrm{diag}}
\def\DC{\mathrm{DC}}
\def\TC{\mathrm{TC}}
\def\SS{\mathrm{SS}}

\def\EB{\mathrm{EB}}

\def\SURE{\mathrm{SURE}}
\def\GCV{\mathrm{GCV}}
\def\GML{\mathrm{GML}}

\def\logdet{\mathrm{logdet}}

\def\rank{\mathrm{rank}}
\def\rmL{\mathrm{L}}
\def\rmR{\mathrm{R}}
\def\rmd{\mathrm{d}}
\def\rmD{\mathrm{D}}

\def\MRef{\mathsf{Ref}}
\def\GR{\mathsf{GR}}
\def\GvR{\mathsf{GvR}}
\def\GvRt{\mathsf{GvRt}}
\def\GRs{\mathsf{GRs}}
\def\SNR{\mathrm{SNR}}
\def\fit{\mathrm{fit}}

\def\High{\mathsf{High}}

\def\norm#1{\lVert #1 \rVert}
\def\abs#1{\left\lvert #1 \right\rvert}

\def\sbk#1{\left( {#1}\right)}
\def\mbk#1{\left[ {#1}\right]}
\def\cbk#1{\left\{ {#1}\right\}}
\def\argmin#1{\mathop{\arg\min}\limits_{#1}}

\ifpdf
\hypersetup{
  pdftitle={Numerically Efficient and Stable Algorithms for Kernel-Based Regularized System Identification using Givens-Vector Representation},
  pdfauthor={Z. Shen, J. Zhang, Martin S. Andersen, and T. Chen}
}
\fi

\begin{document}

\maketitle
\begin{abstract}
Numerically efficient and stable algorithms are essential for kernel-based regularized system identification. 
The state of art algorithms exploit the semiseparable structure of the kernel and are based on the generator representation of the kernel matrix. 
However, as will be shown from both the theory and the practice, the algorithms based on the generator representation are sometimes numerically unstable, which limits their application in practice. 
This paper aims to address this issue by deriving and exploiting an alternative Givens-vector representation of some widely used kernel matrices. Based on the Givens-vector representation, we derive algorithms that yield more accurate results than existing algorithms without sacrificing efficiency.
We demonstrate their usage for the kernel-based regularized system identification.  
Monte Carlo simulations show that the proposed algorithms admit the same order of computational complexity as the state-of-the-art ones based on generator representation, but without issues with numerical stability.
\end{abstract}



\vspace{-1mm}

\section{Introduction}\label{sec:intro}

The class of semiseparable matrices is one of the most widely used classes of rank structured matrices, and has applications in many fields, e.g.\ integral equations \cite{greengard1991numerical}, statistics \cite{greenberg1959matrix,keiner2009fast}, Gaussian process regression \cite{foreman2017fast,andersen2020smoothing}, and \textit{kernel-based regularized system identification} (KRSysId) \cite{andersen2020smoothing,CA21}.  
Specifically, a symmetric matrix $K\in\R^{N\times N}$ is a so-called \textit{(extended) $p$-generator representable semiseparable} ($p$-GRS) (see \Cref{def:gen-p-semi}) with $p\in\mathbb N$, if its lower-triangular part $\tril(K)$ has the form
\begin{align*}
    \tril(K) = \tril(U V^T), \qquad U,V\in\R^{N\times p}, \ p\leq N, 
\end{align*}
where $U,V$ are called the generators of $K$ and comprise its \textit{generator representation} (GR).
As is well known from \cite{Vandebril2008a,GoVa13}, operation with $p$-GRS matrices can be performed efficiently via their GR. For example, given a vector $\x\in\R^N$, the cost of computing $K\x$ can be reduced from $\calO(N^2)$ to $\calO(Np)$ floating-point operations (flops).
For KRSysId \cite{CA21}, the bottleneck of computation is the matrix operations (see \Cref{tab:krm_algorithms}) related to the kernel matrix $K_{\boldeta}\in\R^{N\times N}$ in 
\eqref{eq:Psi_M_H_unitImp} and the regression matrix $M_{\boldeta,\gamma}=\Psi_{\boldeta}+\gamma I_N$ in \cref{eq:sysid_rkhs_regmat}, 
where $\Psi_{\boldeta}\in\R^{N\times N}$ is the output kernel matrix \cref{eq:output_kernel_mat}, $\boldeta\in\R^m$ with $m\in\mathbb N$ is the hyper-parameter, $\gamma>0$ is the regularization parameter, and $I_N$ is the $N$-dimensional identity matrix.
For example, for a given output vector $\y\in\R^{N}$, straightforward computation of $M_{\boldeta,\gamma}^{-1}\y$ in \cref{eq:sysid_rkhs_alpha} and $\y^TM_{\boldeta,\gamma}^{-1}\y$,  $\logdet(M_{\boldeta,\gamma})$, and $\tr(M_{\boldeta,\gamma}^{-1})$ in \cref{eq:EB_opt,eq:SURE_opt,eq:GCV_opt,eq:GML_opt} requires $\calO(N^3)$ flops.
Fortunately, if $\Psi_{\boldeta}$ and $K_{
\boldeta}$ are $p$-GRS, the computational complexity can be reduced to $\calO(Np^3)$ through the GR-based  algorithms derived in \cite{andersen2020smoothing,CA21}.
However, the GR-based algorithms have numerical stability issues. 
To illustrate this, we outline two simple examples below; further details can be found in \Cref{sec:two_eg}.

1. The GR may exhibit diverging scales, i.e., the entries of $U$ and $V$ may grow or decay exponentially even when the entries of $K_{\boldeta}$ are moderate. In finite precision arithmetic, this may lead to overflow/underflow and instability in the computation of $K_{\boldeta}\x$ using GR-based algorithms; see \cite[Example 2]{vandebril2005note}.
        For example, consider the diagonal correlated (DC kernel) \cite{COL2012}, which we define in \cref{eq:gen_DC}, with parameters $c=1$, $\lambda =10^{-1}$, $\rho = 10^{-7}$, $N=5$, and $t_i=i$ for $i=1,\dots,5$. This yields a kernel matrix $K_{\boldeta}^{\DC}=\tril(UV^T)+\tril(VU^T,1)$ with generators 
        \begin{align*} 
          U &= \begin{bmatrix} 10^{-8} & 10^{-16} & 10^{-24} & 10^{-32} & 10^{-40}\end{bmatrix}^T \\
          V &= \begin{bmatrix}10^6 & 10^{12}&10^{18}&10^{24}&10^{30}\end{bmatrix}^T.
        \end{align*}
        When applying \cite[Algorithm 4.1]{andersen2020smoothing} to compute  $\y=K_{\boldeta}\x$,
        the entries of $\y$ span an enormous range, e.g.,
        $y_5 =10^{-40} ([10^6, 10^{12}, 10^{18}, 10^{24}, 10^{30}]^T \x)$.
        If we take $\x=[-1,1,-1,1,-1]^T$ and compute in double precision, then the relative error is of order $10^7$ despite the problem of evaluating $K_{\boldeta}\x$ being well conditioned.

2. Efficient GR-based algorithms for computing the inverse of the Cholesky factor $L_{\boldeta,\gamma}$ of $M_{\boldeta,\gamma}$ face numerical instability when $\gamma > 0$ is small, even if $K_{\boldeta}$ is well conditioned. Given a GR $(U,V)$ of $K_{\boldeta}$, by \cite[Theorem 4.1, Algorithms 4.3 and 4.4]{andersen2020smoothing}, the Cholesky factor $L_{\boldeta,\gamma}$ can be written as $L_{\boldeta,\gamma} = \tril(UW^T,-1)+\diag(\c)$ with GR $(U,W)$, and its inverse can be written as $L_{\boldeta,\gamma}^{-1} = \tril(YZ^T,-1)+\diag(\c)^{-1}$ with GR $(Y,Z)$, where $Y=L_{\boldeta,\gamma}^{-1}U$ and $Z=L_{\boldeta,\gamma}^{-T}W(Y^TW-I_2)^{-1}$. 
To illustrate, consider the stable spline (SS) kernel \cite{COL2012}, which we define in \cref{eq:gen_SS}, with parameters $c=1$, $\rho=0.5$, $N=5$, and $t_i=i$ for $i=1,\dots,5$. The condition numbers of $M_{\boldeta,\gamma}$ and $Y^T W - I_2$ are $\kappa(M_{\boldeta,\gamma})\approx 10^4$ and $\kappa(Y^T W - I_2)\approx 10^{16}$, respectively, leading to an inaccurate GR $(Y,Z)$ with the relative error of $Z$ being of order 1 in double precision.
Moreover, even with an accurate GR $(Y,Z)$, we may not be able to accurately compute, for $1\leq j < i \leq 5$, the $(i,j)$-entry $\y_i^T\z_j$ of $\tril(L_{\boldeta,\gamma}^{-1},-1)$, where $\y_i,\z_j\in\R^2$ are the $i$th and $j$th column of $Y^T$ and $Z^T$, respectively.
The reason is that the relative condition numbers \cite[Section~3]{higham2002accuracy} associated with the inner product $\y_i^T\z_j$  when computing $\tril(L_{\boldeta,\gamma}^{-1},-1)$ can be extremely large, up to order $10^{16}$; see \cref{eq:two_eg_cond_inner} for all the relative condition numbers. 
As a result, even if $(Y,Z)$ is accurate to double precision, the entries of $\tril(L_{\boldeta,\gamma}^{-1},-1)$ cannot be reliably computed.
Indeed, if we compute $(Y,Z)$ to 50 decimal digits of accuracy and round them to double precision, then the relative error of $\tril(L_{\boldeta,\gamma}^{-1},-1)$ is approximately $1.95$.

The above examples suggest that the numerical instability of the GR-based algorithms \cite{CL13,CA21,SXAC23,CCDA23} may limit their application in practice. 
To address this issue, we employ the Givens-vector representation (GvR) of $p$-semiseparable kernels (see \Cref{def:semi}) \cite{vandebril2005note,Vandebril2008a,GoVa13}, which offers better numerical stability, but the corresponding algorithms are generally more tedious to derive. 
To this end, we first derive the GvR for some widely used kernel matrices and their corresponding output kernel matrices for the KRSysId.
We then derive the GvR-based implementation of algorithms with computational complexity of $\calO(Np^2)$ flops.
In particular, we provide the GvR for the Cholesky factor $L_{\boldeta,\gamma}$ of $M_{\boldeta,\gamma}$ in \cref{eq:chol_L} as well as an implicit representation of $L_{\boldeta,\gamma}^{-1}$ in \cref{eq:chol_Linv,eq:inv_rep}. Notably, revisiting the two examples above by using GvR-based algorithms shows that, the relative errors of computing $\y=K_{\boldeta}\x$ via \Cref{alg:Ax} and reconstructing $\tril(L_{\boldeta,\gamma}^{-1},-1)$ via \cref{eq:chol_Linv,eq:inv_rep} are of orders $10^{-8}$ and $10^{-11}$, respectively, in double precision.
Moreover, we find a different route to compute $\tr(M_{\boldeta,\gamma}^{-1})$ with $\calO(Np^2)$ flops.
In contrast, the algorithm proposed in \cite{andersen2020smoothing} requires computing the implicit representation of $L_{\boldeta,\gamma}^{-1}$ and takes  $\calO(Np^3)$ flops. 
Finally, we apply our proposed GvR-based algorithms for the KRSysId, and we show through numerical simulations that our proposed implementation outperforms the state-of-the-art GR-based ones in both the numerical stability and efficiency.

In \Cref{sec:background}, we briefly review the KRSysId theory under \textit{reproducing kernel Hilbert space} (RKHS) framework. 
In \Cref{sec:semiseparable}, we introduce semiseparable matrices, GR, and GvR, and derive the GvR of some commonly used kernel matrices and output kernel matrices in the KRSysId. In \Cref{sec:alg}, we provide  GvR-based efficient implementation of algorithms.
In \Cref{sec:experiments}, we run numerical simulations to illustrate the numerical stability, efficiency and accuracy of the proposed algorithms, and finally, we conclude the paper in \Cref{sec:conclusions}.

\section{Background and related work}\label{sec:background}
In this section, we first briefly list notations used throughout the paper, then introduce some necessary background  materials about the KRSysId, and finally,  we review the state-of-the-art implementation of algorithms in the KRSysId.

\subsection{Notation}
Let $\R_+=[0,\infty)$ and $\N=\{1,2,\cdots\}$.
Let $\R_+^N$, and $\R_{++}^N$ be the set of nonnegative $N\times 1$ real vectors, and $N\times 1$ elementwise positive vectors, respectively. 
Let $\e_i\in\R^N$ be the vector of zeros except $1$ in the $i$th entry. 
Denote $\1_N$, $\0_N$, and $\0_{N\times m}$ the $N\times 1 $ vector of ones, $N\times 1 $ vector of zeros, and $N\times m$ matrix of zeros, respectively, where the subscript $N$ is dropped when there is no confusion. 
For vector $\x\in\R^N$, let $x_i$ be the $i$th element of $\x$.    
For $A\in\R^{N\times N}$, denote $A_{i,j}$ or $A(i,j)$ its $(i,j)$-entry, $A(i:j,i':j')$ the MATLAB-like sub-block of $A$ from the $i$th row to the $j$th row and from the $i'$th column to the $j'$th column. 
Denote $\tril(A,k)$ ($\triu(A,k)$) the matrix with all elements above (below) the $k$th superdiagonal being zero and let $\tril(A)=\tril(A,0)$ and $\triu(A)=\triu(A,0)$. 
Given $\d\in\R^N$, $\diag(\d)\in\R^{N\times N}$ is a diagonal matrix with $\d$ as its diagonal part. 
Let $\1(\cdot)$ be the indicator function. 
For $A\in\R^{N\times N}$, $A\succ\0$ means that $A$ is positive definite.   
For sequence $\{S_i\}_{i\in\N}$ where $S_i\in\R^{p\times p}$ for $p\in\N$, define the multiple product $S_{i:j}^{>}=\prod_{k=0}^{i-j}S_{i-k}$ for $i\geq j$ and $I_p$ for $i<j$.

\subsection{Kernel-based regularized system identification (KRSysId)}\label{sec:krm}

In the past decade, the \textit{kernel-based regularized method} has emerged and gradually become a new paradigm for system identification \cite{LCM20}. 

Consider a linear time-invariant (LTI), causal, and stable system described by 
\begin{align}\label{eq:LTI}
    y(t)=(g*u)(t)+\varepsilon(t), \quad t\geq 0,
\end{align}
where $y(t)\in\R$, $u(t)\in\R$, $g(t)$, and $\varepsilon(t)\in\R$ are called the measurement output, input, impulse response, and disturbance of the system at time $t$, respectively, and the convolution of $g$ and $u$ is defined as
\begin{align*}
   (g*u)(t)=\left\{\begin{array}{cc}
    \sum_{\tau=0}^\infty g(\tau)u(t-\tau),\ t\in \{0\}\cup \N, & \quad \text{discrete-time (DT)}, \\
    \int_0^\infty g(\tau)u(t-\tau) d\tau,\ t\in\R_+, & \quad \text{continuous-time (CT)}. 
    \end{array}\right.
\end{align*}
$\varepsilon(t)$ is assumed to be independent and identically distributed (i.i.d) with mean zero and variance $\sigma^2$ and independent of $u(t)$. The goal of identification is to estimate the impulse response $g(t)$ based on $\y=[y(t_1)\ \cdots \ y(t_N)]^T$, and $\{u(t):t\in\R_+\}$ for the CT case, and $\{u(t_i):i\in\{0\}\cup \N \}$ for the DT case with $t_i=i$. In the calculation of $(g*u)(t)$, it is common to assume that $u(t)=0$ when $t<0$.

The KRM can be equivalently formulated in a couple of different ways \cite{pillonetto2014kernel,PCCDL22}. Here, it is
formulated as a function estimation problem in an RKHS determined by a positive semidefinite \textit{kernel} function.
To be specific, we need to first recall some definitions in relation to RKHS.
An RKHS $\calH$ over a nonempty function domain $\calX$ is a Hilbert space of functions $f:\calX\to\R$ equipped with norm $\norm{\cdot}_{\calH}$ such that all the evaluators $f\mapsto f(x)$ are linear and bounded over $\calH$ \cite{aronszajn1950theory,wahba1990spline,halmos2017introduction}.
It can be shown that there exists a unique positive semidefinite kernel $\calK:\calX\times\calX\to\R$ such that $\calK(x,\cdot)\in\calH$, $\calK(x_i,x_j)=\calK(x_j,x_i)$, and $\sum_{i,j=1}^N a_i a_j \calK(x_i,x_j)\geq 0$ for any $N\in\N$, $x_i,x_j\in\calX$ and $a_i,a_j\in\R$, and moreover, the following so-called reproducing property holds: $\inp{\calK(x,\cdot)}{f}_{\calH}=f(x)$ for all $(x,f)\in(\calX,\calH)$,
where $\langle\cdot,\cdot\rangle_{\calH}$  is the inner product of $\calH$ \cite{aronszajn1950theory}.
Conversely, given a positive semidefinite kernel $\calK:\calX\times\calX\to\R$, it can be shown by the Moore--Aronszajn theorem \cite{aronszajn1950theory} that 
there exists a unique RKHS on $\calX$ for which $\calK$ is its reproducing kernel.

For KRM, we first assume that a positive semidefinite kernel $\calK(t,s;\boldeta)$ has been carefully designed to embed the prior knowledge of the underlying system to be identified, where $\boldeta\in\R^m$ is a hyper-parameter vector. In the DT case, we take $\calX=\{0\}\cup \N$, and in the CT case, we take $\calX=\R_+$. 
Then, we let $\calH$ be the RKHS induced by this kernel and take $\calH$ to the hypothesis space in which we will search for the impulse response $g$. Furthermore, we estimate
the impulse response $g$ by minimizing the following regularized least squares criterion
\begin{align}\label{eq:sysid_rkhs}
   \hat{g}=\argmin{g\in\calH } \sum_{i=1}^{N}(y(t_i)-(g*u)(t_i))^2+\gamma \norm{g}_{\calH}^{2},
\end{align}
where $L_t[g]=(g*u)(t)$ is a linear and bounded functional $L_t \colon \calH\to\R$, $\norm{ \cdot }_{\calH}$ is the norm of $\calH$, and $\gamma>0$ is a regularization parameter, which is also regarded as a hyper-parameter.
The representer theorem \cite{wahba1990spline,pillonetto2014kernel,PCCDL22}
shows that  the solution of \cref{eq:sysid_rkhs} has the form
\begin{align}
  \hat{g}(t)=\sum_{i=1}^N \hat{\alpha}_i\bar{a}_i(t;\boldeta) \label{eq:sysid_rkhs_g}, \quad \hat{\boldalpha}=[\hat{\alpha}_1 \ \cdots \ \hat{\alpha}_N]^T,
\end{align}
for some coefficients $\hat{\alpha}_i\in\R$ and the representer $\bar{a}_i$ of $L_{t_i}$ with $L_{t_i}[g]=\inp{g}{\bar{a}_i}_{\calH}$ for all $i$ and $g\in\calH$, and
\begin{align*}
   \bar{a}_i(t;\boldeta)&=\bar{a}(t,t_i;\boldeta)=L_{t_i}[\calK(\cdot,t;\boldeta)]=(\calK(t,\cdot;\boldeta)*u)(t_i) \\
  &=\begin{cases}
       \sum_{\tau=0}^\infty\calK(t,\tau;\boldeta)u(t_i-\tau),& t,t_i\in\{0\}\cup \N, \  \ \quad\text{(DT)},\\
       \int_0^\infty\calK(t,\tau;\boldeta)u(t_i-\tau)\rmd\tau,& t,t_i\in\R_+, \quad\text{(CT)},
   \end{cases}
\end{align*}
By the relation
\begin{align*}
    L_{t_i}[L_{t_j}[\calK]]=L_{t_i}[\bar{a}_j]&=\inp{\bar{a}_j}{\bar{a}_i}_{\calH}=\inp{\bar{a}_i}{\bar{a}_j}_{\calH}=L_{t_j}[\bar{a}_i]=L_{t_j}[L_{t_i}[\calK]],
\end{align*}
plugging \cref{eq:sysid_rkhs_g} into \cref{eq:sysid_rkhs} gives an equivalent problem of \cref{eq:sysid_rkhs} as follows
\begin{align}\label{eq:sysid_rkhs_regmat0}
\hat{\boldalpha}=\argmin{\boldalpha\in\R^N}\sum_{i=1}^N\sbk{
    y(t_i)-\sum_{j=1}^N \alpha_j\inp{\bar{a}_{j}}{\bar{a}_{i}}_{\calH}
    }^2 + \gamma \sum_{i=1}^N \sum_{j=1}^N \alpha_i \alpha_j \inp{\bar{a}_{i}}{\bar{a}_{j}}_{\calH}.
\end{align}
Let the \textit{output kernel matrix} and \textit{output kernel} \cite{pillonetto2010new,pillonetto2014kernel,CA21,PCCDL22} be
\begin{align}
    \Psi_{\boldeta}&=(\Psi(t_i,t_j;\boldeta))_{1\leq i,j\leq N}=(\inp{\bar{a}_i}{\bar{a}_j}_{\calH})_{1\leq i,j\leq N}, \label{eq:output_kernel_mat} \\
    \Psi(t,t';\boldeta)&=\begin{cases}
        \sum_{s=0}^{\infty}\sum_{r=0}^{\infty}\calK(s,r;\boldeta)u(t-s)u(t'-r),\ t,t'\in\{0\}\cup \N,&\ (\text{DT}),\\
        \int_0^\infty\int_0^\infty \calK(s,r;\boldeta)u(t-s)u(t'-r)\rmd r\rmd s,\ t,t'\in\R_+&\ (\text{CT}),
    \end{cases}\label{eq:output_kernel}
\end{align}respectively.
Then \eqref{eq:sysid_rkhs_regmat0} becomes
\begin{align}
  \hat{\boldalpha}
  &=\argmin{\boldalpha\in\R^N}\ \norm{\y-\Psi_{\boldeta}\boldalpha}_2^2+\gamma \boldalpha^T\Psi_{\boldeta}\boldalpha=M_{\boldeta,\gamma}^{-1}\y, \label{eq:sysid_rkhs_alpha} \\
  M_{\boldeta,\gamma}&=\Psi_{\boldeta}+\gamma I_N.\label{eq:sysid_rkhs_regmat}
\end{align} 
Then, we have the fitted values $\hat{\y}=\Psi_{\boldeta}\hat{\boldalpha}=H_{\boldeta,\gamma}\y$, where $H_{\boldeta,\gamma}=\Psi_{\boldeta}M_{\boldeta,\gamma}^{-1}$ is the so-called \textit{influence matrix}, and the predicted output at time $t$ $\hat{y}(t)=(\hat{g}*u)(t)=\sum_{i=1}^N \hat{\alpha}_i\Psi(t,t_i;\boldeta)$.

It is interesting to note that the KRM includes the function estimation in RKHS, which is widely studied in the field of machine learning and statistics, e.g., \cite{wahba1990spline}, as a special case, when considering $u(t)$ to be the unit impulse signal. 
\begin{example}[Function estimation in RKHS]\label{ex:unitImp}
When $u(t)$ is the unit impulse signal, i.e., $u(t)$ is the Dirac delta for CT case and $u(t)=\1(t=0)$ for DT case, we have
$\Psi(t,t';\boldeta)=\calK(t,t';\boldeta)$ and $\hat{y}(t)=\hat{g}(t)$ for all $t,t'$,  then the model \cref{eq:LTI} and the regularized least squares criterion \cref{eq:sysid_rkhs} becomes
\begin{align*}
    y(t)&=g(t)+\varepsilon(t), \quad t\geq 0, \\
    \hat{g}&=\argmin{g\in\calH } \sum_{i=1}^{N}(y(t_i)-g(t_i))^2+\gamma \norm{g}_{\calH}^{2}.
\end{align*}
Let $K_{\boldeta}=(\calK(t_i,t_j;\boldeta))_{1\leq i,j\leq N}$ be the \textit{kernel matrix}, by
\begin{align}\label{eq:Psi_M_H_unitImp}
    \Psi_{\boldeta}=K_{\boldeta}, \ 
    M_{\boldeta,\gamma}=K_{\boldeta}+\gamma I_N, \ 
    H_{\boldeta,\gamma}=K_{\boldeta}M_{\boldeta,\gamma}^{-1},
\end{align}
the solution \eqref{eq:sysid_rkhs_alpha} become
\begin{align*}
    \hat{\boldalpha}&=\argmin{\boldalpha\in\R^N}\ \norm{\y-K_{\boldeta}\boldalpha}_2^2+\gamma \boldalpha^TK_{\boldeta}\boldalpha=M_{\boldeta,\gamma}^{-1}\y.
\end{align*}In particular, when $\calK(t_i,t_j;\boldeta)$ is taken to be the spline kernel,  the function estimation problem further becomes a special case of the so-called \textit{smoothing spline regression problem} without the inclusion of basis functions \cite{wahba1975smoothing,wahba1990spline}.
\end{example}

\subsection{Kernels and hyper-parameter estimation}\label{sec:hyper_parameter_est}
From a theoretical perspective, the major difficulty of KRM lies in the design of a suitable kernel $\calK(t,s;\boldeta)$ and also in the estimation of the hyper-parameters $(\gamma,\boldeta)$.
The issue of kernel design has attracted a lot of interests in the past decade, e.g., \cite{chen2018kernel,ZC18,CP18,BP2020,FC2024,Zorzi2024}. 
Commonly used kernels include the stable spline (SS) kernel \cite{COL2012}, the diagonal correlated  (DC) kernel, and the tuned-correlated (TC) kernel \cite{COL2012}:
\begin{subequations}
\begin{align}
  \calK^{\SS}(t,s;\boldeta^{\SS})&= c\frac{\rho^{(t+s)+\max\{t,s\}  } }{2} - c \frac{\rho^{ 3\max\{t,s\} }}{6}, \quad \boldeta^\SS=(c,\rho)\in \R \times (0,1),  \label{eq:SS}\\
  \calK^{\DC}(t,s;\boldeta^{\mathrm{DC}})&=c \lambda^{t+s}\rho^{\abs{t-s}}, \quad \boldeta^{\DC}=(c,\lambda,\rho)\in \R\times (0,1] \times (0,1), \label{eq:DC}\\
\calK^{\TC}(t,s;\boldeta^{\TC})&=c \rho^{(t+s)+\abs{t-s}},  \hspace*{1.5em} \boldeta^{\TC}=(c,\rho)\in \R\times (0,1) .\label{eq:TC}
\end{align}
\end{subequations}
Note that 
$\calK^{\TC}$ is a special case of $\calK^{\DC}$ by letting $\lambda=\rho$.

The issue of hyper-parameter estimation can be done by minimizing different criteria with respect to  hyper-parameter $(\gamma,\boldeta)$. Four widely used criteria are the empirical Bayes (EB), Stein's unbiased risk estimation (SURE) \cite{mu2018asym}, generalized cross validation (GCV) \cite{Golub1979,wahba1990spline}, and generalized maximum likelihood (GML) \cite{zhang2024asymptotic}, where the objectives are
\begin{subequations}
\begin{align}
    \EB(\gamma,\boldeta) &
  = \y^{T} M_{\boldeta,\gamma}^{-1} \y + \log\det(M_{\boldeta,\gamma}), \label{eq:EB_opt} \\
    \SURE(\gamma,\boldeta) &=  
    \norm{\y-\hat{\y}}^2 + 2\gamma \tr(H_{\boldeta,\gamma}) 
  ,\label{eq:SURE_opt} \\
      \GCV(\gamma,\boldeta) &=  \frac{\norm{\y-\hat{\y}}^2}{(1-\tr(H_{\boldeta,\gamma})/N)^2}=\frac{N^2\norm{\y-\hat{\y}}^2}{(\gamma\tr(M_{\boldeta,\gamma}^{-1}))^{2}},\label{eq:GCV_opt} \\
      \GML(\gamma,\boldeta)&= N\log(\y^T M_{\boldeta,\gamma}^{-1}\y) +\log\det(M_{\boldeta,\gamma})-N\log N, \label{eq:GML_opt}
\end{align}
\end{subequations}
where $1-\tr(H_{\boldeta,\gamma})/N=\gamma \tr(M_{\boldeta,\gamma}^{-1})/N$ is by matrix inversion lemma.
The kernel scaling factor $c$ and the noise variance $\sigma^2$ can be absorbed into the regularization parameter as $\gamma=\sigma^2/c$. Therefore, from a computational perspective, it suffices to take $c=1$ and consider the case presented in \Cref{sec:krm} with the regularization parameter $\gamma$.

From a practical perspective, the major difficulty of KRM lies in the computation of $\hat{\boldalpha}$ in \cref{eq:sysid_rkhs_alpha} and 
\cref{eq:EB_opt,eq:SURE_opt,eq:GCV_opt,eq:GML_opt}, summarized in \Cref{tab:krm_algorithms}. 
Since they include terms
$M_{\boldeta,\gamma}^{-1}\y$, 
$\y^T M_{\boldeta,\gamma}^{-1}\y=\y^T\hat{\boldalpha}$, $\logdet(M_{\boldeta,\gamma})$, 
$\y-\hat{\y}=\y-\Psi_{\boldeta}\hat{\boldalpha}$, $\tr(H_{\boldeta,\gamma})$, and $\tr(M_{\boldeta,\gamma}^{-1})$, 
a straightforward computation  requires $\calO(N^3)$ computational flops. 
Clearly, this is prohibitively expensive for large $N$ and thus it is interesting and important to develop efficient and stable algorithms to compute \cref{eq:sysid_rkhs_alpha} and 
\cref{eq:EB_opt,eq:SURE_opt,eq:GCV_opt,eq:GML_opt}.

\subsection{Related work}
There are two classes of numerically efficient implementation of algorithms for KRM:
the optimization based ones \cite{CL13,CA21,SXAC23,CCDA23,XFMC2024} and the full Bayesian one \cite{PL2023}. 
Here, we are interested in the first class, and the bottleneck is the computation of the hyper-parameter estimation criteria \cref{eq:EB_opt,eq:SURE_opt,eq:GCV_opt,eq:GML_opt}.
By assuming the FIR model with model order $n$,
\cite{CL13} proposed an algorithm with complexity $\calO(Nn^2+n^3)$ that avoids explicit matrix inversion by employing the QR factorization. In \cite{SXAC23}, by assuming the FIR model and the periodic input signal with period $q$, 
an algorithm with complexity $\calO(Nq+q^3+nqp'+nq^2)$ was proposed by exploiting the GR of $K_{\boldeta}$ and the hierarchically
semiseparable structure (HSS) of $\Psi_{\boldeta}$ \cite{Massei2020toolbox}, where $p'$ is the semiseparability rank of the kernel.
In \cite{CCDA23}, an algorithm with complexity $\mathcal{O}((N+n)\log(N+n)+nl^2)$ (provided that the number of function evaluations in the Bayesian optimization loop is fixed) was proposed, where $l$ is the rank of the randomized Nyström approximation.
The algorithm exploits the GR of $K_{\boldeta}$, and leverages stochastic trace estimation to compute $\log\det(M_{\boldsymbol{\eta},\gamma})$ and an iterative solver such as LSQR to compute $M_{\boldsymbol{\eta},\gamma}^{-1}\mathbf{y}$.

In \cite{CA21}, by assuming a class of widely used test input signals in system identification and automatic control, and considering model \cref{eq:LTI}, an algorithm with complexity $\calO(Np^3)$ was proposed by exploiting the GR of $K_{\boldeta}$ and $\Psi_{\boldeta}$, and leveraging the algorithms in \cite{Vandebril2008a,andersen2020smoothing}, where $p$ is the semiseparability rank of $\Psi_{\boldeta}$.
In \cite{XFMC2024}, by considering the frequency response model, an algorithm with complexity $\calO(r^2N(\log(N))^2)$ was proposed by exploiting the hierarchically off-diagonal low-rank (HODLR) structure of the output kernel matrix, where $r$ is the HOLDR rank \cite{Massei2020toolbox}. 

The above implementations, except \cite{XFMC2024}, are based on the GR of the kernel matrix and rely on GR-based algorithms,
which are numerically unstable in some cases \cite{vandebril2005note}, as illustrated in \Cref{sec:intro}.
The GvR-based algorithms to be introduced in the next section can effectively overcome this issue;
see \cite{vandebril2005note,vandebril2007matrix,Vandebril2008a} for a comprehensive overview.

\section{Semiseparable matrices}\label{sec:semiseparable}
\subsection{Generator representation}\label{sec:gen-givens}
The original definition of semiseparable matrices is the inverse of irreducible tridiagonal matrices (i.e., the subdiagonal elements are non-zero), which is also called one-pair matrix \cite{gantmacher2002oscillation,vandebril2005note}. 
Another commonly used definition of semiseparable matrices is based on generators \cite{chandrasekaran2000fast,van2004two,vandebril2005note,Vandebril2008a}, extending the semiseparability rank from $1$ in one-pair matrices to be higher than $1$, defined below, where we mainly focus on the symmetric case.
\begin{definition}\label{def:gen-p-semi}
Let $p\in\N$, a symmetric matrix $A\in\R^{N\times N}$ is said to be (extended) $p$-generator representable semiseparable ($p$-GRS) if 
\begin{align}
    A =  \tril(UV^T) + \triu(VU^T,1) \label{eq:gen-p-sym-semi},
\end{align}
where $U=\begin{bmatrix}
        \boldmu_1 & \cdots \boldmu_N
    \end{bmatrix}^T, \ 
    V=\begin{bmatrix}
        \boldnu_1 & \cdots & \boldnu_N
    \end{bmatrix}^T\in\R^{N\times p}$ with
$\boldmu_i,\boldnu_i\in\R^{p}$ (when $p=1$, we let $\boldmu_i=\mu_i,\ \boldnu_i=\nu_i$) are called \textit{generators} of $A$. 
For a general matrix $A\in\R^{N\times N}$, it is said to be (extended) $\{p,q\}$-GRS if $A=\tril(UV^T)+\triu(PQ^T,1)$ for some $U,V\in\R^{N\times p}$ and $P,Q\in\R^{N\times q}$.  
The $(i,j)$-entry of $A$ can be represented as 
\begin{equation}\label{eq:gen_mat}
    \begin{aligned}
    A(i,j)&=\begin{cases}
        \boldmu_i^T\boldnu_j &\text{ if } 1\leq j\leq i\leq N,\\
        \boldmu_j^T\boldnu_i &\text{ if } 1\leq i<j\leq N.
      \end{cases} 
    \end{aligned}
\end{equation}
\end{definition}
Let $\calG_{N,p}$ be the class of $N\times N$ symmetric $p$-GRS matrices. Then for any $A\in\calG_{N,p}$, $A$ and matrices in the form of $A$-plus-diagonal allow cheap memory storage and fast algorithms by exploiting its GR, such as matrix-vector product \cite{Vandebril2008a}, QR decomposition-based linear system solver \cite{van2004two}, and matrix inversion \cite{gantmacher2002oscillation}.

For KRM, let $K_{\boldeta}^{\SS}$, $K_{\boldeta}^{\DC}$, and $K_{\boldeta}^{\TC}$ denote the kernel matrices of SS, DC, and TC kernels, respectively.
It was shown in \cite[Proposition 2]{CA21}  that $K_{\boldeta}^{\SS}\in\calG_{N,2}$, $K_{\boldeta}^{\DC}\in\calG_{N,1}$, and $K_{\boldeta}^{\TC}\in\calG_{N,1}$ with GR 
\begin{subequations}
\begin{align}
    (\SS)\quad &\boldmu_i=\begin{bmatrix}
        -{\rho^{3 t_i}}/{6} & {\rho^{2 t_i}}/{2}
    \end{bmatrix}^T, \quad \boldnu_j=\begin{bmatrix}
            1 & \rho^{ t_j}
        \end{bmatrix}^T, \label{eq:gen_SS} \\
    (\DC)\quad &\mu_i=(\lambda\rho)^{t_i}, \quad \nu_j=(\lambda/\rho)^{t_j}, \label{eq:gen_DC} \\
    (\TC)\quad &\mu_i=(\rho)^{2t_i}, \quad \nu_j=1. \label{eq:gen_TC}
\end{align}    
\end{subequations}
The structure of the output kernel matrix $\Psi_{\boldeta}$ depends not only on the kernel, but also on the choice of the input. 
In Example~\ref{ex:unitImp}, $\Psi_{\boldeta}=K_{\boldeta}$ by \cref{eq:Psi_M_H_unitImp}, so given $K_{\boldeta}\in\calG_{N,p'}$, we have $\Psi_{\boldeta}\in\calG_{N,p'}$. 
More generally, if $K_{\boldeta}\in\calG_{N,p'}$ and the input $u(t)$ satisfies 
\begin{align}\label{eq:input_condition_GR_output_kernel}
 u(t-b)=\sum_{k=1}^r \pi_k(t) \rho_k(b), \ \pi_k,\rho_k:\R_+\to \R, \ r\in\N,
\end{align}
then $\Psi_{\boldeta}\in\calG_{N,p}$ with $p=p'+r$ by \cite[Theorem 3]{CA21}.
The condition \cref{eq:input_condition_GR_output_kernel} is mild, and many commonly used test input signals in automatic control satisfy this condition including
\begin{subequations}
  \begin{align}
    \text{(Polynomial)} \quad u(t)&=t^q, \ q\in\N, \label{eq:input-poly}\\
    \text{(Sinusoidal)} \quad u(t)&=\sin(\omega t + \theta), \ \omega,\theta\in\R, \label{eq:input-sin}\\
    \text{(Exponential)} \quad u(t)&=e^{-\beta t}, \quad \beta\in\R, \label{eq:input-exp}
  \end{align}
\end{subequations}their products, and their linear combinations.
In particular, we have $r=q+1$ for \cref{eq:input-poly} and $r=2$ for \cref{eq:input-sin}.  
As long as the GR of $\Psi_{\boldeta}$ are available, the fast algorithms derived in \cite{andersen2020smoothing,CA21} can be directly applied to $\Psi_{\boldeta}$ to calculate \cref{eq:sysid_rkhs_alpha} and \cref{eq:EB_opt,eq:SURE_opt,eq:GCV_opt,eq:GML_opt} in at most $\calO(Np^3)$ flops.
\Cref{prop:Giv_output} provides an example for the GR \cref{eq:output_kernel_input_exp_GR} of $\Psi_{\boldeta}\in\calG_{N,2}$ with the DC kernel \cref{eq:DC} and the exponential input \cref{eq:input-exp}.

\subsection{Givens-vector representation}
\Cref{def:gen-p-semi} is strong, as the inverse of general tridiagonal matrices may not have a GR. Besides, even though $A$ has a GR, sometimes with finite precision, the reconstruction of $A$ and its relative arithmetical operations are numerically unstable and lose significant digits, e.g., when the difference of number magnitude between $U$ and $V$ are extremely large; see \cite[Example 2]{vandebril2005note}. 

Such cases appear when there is a nearly zero element in the off-diagonal part. 
Let $p=1$, a simple observation reveals that a symmetric $A\in\calG_{N,1}$ has GR \cref{eq:gen-p-sym-semi} if and only if the following statement is true: if $1\leq j\leq i\leq N$ such that $A(i,j)$ vanishes, then $A(i,1:i)=0$ or $A(j:N,j)=0$ \cite{van2005orthogonal}. 
Numerically, if an close-to-zero entry exists in $(i,j)$ for $1\leq j\leq i\leq N$, then either $\mu_i$ or $\nu_j$ must extremely approach zero, which means that during the GR construction, $\nu_{j+1},\ldots,\nu_i$ or $\mu_{i+1},\ldots,\mu_N$ might attain an extremely large magnitude to compensate for small $\mu_i$ or $\nu_j$, making $A(i,j+1),\ldots,A(i,i)$ or $A(i+1,j),\ldots,A(N,j)$ far away from begin numerically vanished, if needed. 

To overcome the drawbacks, a more general class of semiseparable matrices is defined in terms of submatrices rank \cite{vandebril2005note,van2005orthogonal,Vandebril2008a}.

\begin{definition}[$p$-semiseparable]\label{def:semi}
    A symmetric matrix $A\in\R^{N\times N}$ is called a $p$-semiseparable matrix with semiseparability rank $p$ if for $i=1,\ldots,N$, 
    \begin{enumerate}
        \item $\rank(A(i:N,1:i))\leq p$; and
        \item there exists at least one $i,j$ such that $\rank(A(i:N,1:i))= p$.
    \end{enumerate}
If $A$ is lower triangular and items 1--2 holds, then $A$ is called a lower triangular $p$-semiseparable matrix.
\end{definition}
Denote $\calS_{N,p}$ the class of $N\times N$ symmetric $p$-semiseparable matrices.
We have $\calG_{N,p}\subset\calS_{N,p}$, and $\calS_{N,p}$ also includes other easily-expressed matrices, such as diagonal matrices. 
In \cite{vandebril2005note}, the \textit{Givens-vector representation} (GvR) for $A\in\calS_{N,p}$ was proposed to identically represent this wider class of semiseparable matrices, based on which numerically stable algorithms can be derived. 
Specifically, for $A\in \calS_{N,1}$, the idea is to represent $A$ with $N-1$ Givens transformations  and a vector of length $N$, which are called the GvR of $A$. 
For $A\in\calS_{N,p}$, by \cite[Theorem 8.71]{Vandebril2008a}, we can first rewrite $A=\sum_{k=1}^{p} A_k$ for some $A_k\in\calS_{N,1}$, $k=1,\dots,p$. 
Then for $k=1,\dots,p$, the GvR of $A_k$ is given by the following $(N-1)$ nontrivial Givens transformation $\{G_{i,k}\}_{i=1}^{N-1}$ and $\{\hv_{i,k}\}_{i=1}^N$:
\begin{align*}
  G_{i,k}=\begin{bmatrix}
        c_{i,k} & -s_{i,k} \\ s_{i,k} & c_{i,k}
    \end{bmatrix}, \quad \hv_{i,k}\in\R,
\end{align*}where $c_{i,k}^2+s_{i,k}^2=1$ for $i=1,\ldots,N-1$, such that $A_k$ has $(i,j)$-entry $c_{i,k} s_{i-1:j,k}^> \hv_{j,k}$ for $1\leq j\leq i\leq N$ and $\triu(A_k,1)$ can be computed by symmetry. 
We set $c_{N,k}=1$, and $s_{N,k}=0$ for $k=1,\ldots,p$ \cite{vandebril2005note,Vandebril2008a}.
Let $\c_i=(c_{i,1} , \ldots, c_{i,p})^T$, 
$\s_i=(s_{i,1} , \ldots , s_{i,p})^T$, 
$\hbv_i=(\hv_{i,1} , \ldots , \hv_{i,p})^T$, and
$S_i=\diag(\s_i)$ for $i=1,\ldots,N$, then the GvR of $A$ is
\begin{align}\label{eq:giv_mat_p}
  A(i,j) = 
  \begin{cases}
    \c_i^T S_{i-1:j}^{>} \hbv_j &\text{ if } 1\leq j\leq i\leq N,\\
    \c_j^T S_{j-1:i}^{>} \hbv_i &\text{ if } 1\leq i<j\leq N.
  \end{cases} 
\end{align}For $p=1$, we simply write $\c_i$, $\s_i$, and $\hbv_i$ as $c_i$, $s_i$, and $\hv_i$, respectively.

The construction and retrieving procedure for GvR is detailed in \cite{vandebril2005note,Vandebril2008a}, but we skip it, since they cannot be completed in $\calO(N)$ complexity in the most general case. 
We present only the conversion from GR \cref{eq:gen_mat} to GvR here.
Suppose $A_k\in\calG_{N,1}$ for $k=1,\ldots,p$ with GR $U_k=(\mu_{1,k},\ldots,\mu_{N,k})^T$ and $V_k=(\nu_{1,k},\ldots,\nu_{N,k})^T$, then
\begin{subequations}
\begin{align}
  & G_{N-1,k}\begin{bmatrix}
    r_{N-1,k} \\ 0
  \end{bmatrix} = \begin{bmatrix}
    \mu_{N-1,k} \\ \mu_{N,k}
  \end{bmatrix}, \ G_{\ell,k} \begin{bmatrix}
    r_{\ell,k} \\ 0
  \end{bmatrix}=\begin{bmatrix}
    \mu_{\ell,k} \\ r_{\ell+1,k}
  \end{bmatrix},\quad \ell=N-1,\ldots,1, \label{eq:gen2giv_G} \\
  & \abs{\hv_{i,k}}=\abs{\nu_{i,k}}r_{i,k}, \quad r_{i,k}=\sqrt{ \sum_{j=i}^N \mu_{j,k}^2}, \quad i=N,\ldots,1, \label{eq:gen2giv_vhat}
\end{align}
\end{subequations}
where $c_{i,k}\hv_{i,k}$ and $\mu_{i,k} \nu_{i,k}$ have the same signs.
Particularly, if there are $i=N-1,\ldots,1$ such that $r_{i,k}=0$, then we let $c_{i,k}=1$ and $s_{i,k}=0$; and
if $c_{i,k}=0$, then we let $s_{i,k}=1$.
Intuitively, the GvR construction factorizes $U_k$ into products $c_{i,k} s_{i-1:j,k}^>$ with bounded components $c_{i,k}, s_{i,k} \in [-1,1]$.
This procedure is stable as it consists of Givens rotation, with computational complexity $\calO(N)$ \cite{vandebril2005note,Vandebril2008a}. 

For KRM, by GR \cref{eq:gen_SS,eq:gen_DC,eq:gen_TC} and the procedure \cref{eq:gen2giv_G,eq:gen2giv_vhat}, we can obtain the GvR of the kernel matrices $K_{\boldeta}^{\SS}$, $K_{\boldeta}^{\DC}$, and $K_{\boldeta}^{\TC}$.

\begin{proposition}\label{prop:Giv_SS_DC_TC}
    The kernel matrix $K_{\boldeta}^{\SS}\in\calS_{N,2}$ with $c=1$ has GvR
    \begin{equation}\label{eq:Giv_SS}
        \begin{aligned}
         & \c_{i}^T=\begin{bmatrix}
            \frac{-\rho^{3 t_i}}{\sqrt{\sum_{j=i}^{N}\rho^{6 t_j}}} & \frac{\rho^{2 t_i}}{\sqrt{\sum_{j=i}^N \rho^{4 t_j}}}
        \end{bmatrix},\\
        & \s_{i}^T=\begin{bmatrix}
            \frac{(-1)^{\1(i=N-1)} \sqrt{\sum_{j=i+1}^N \rho^{6 t_j}}}{\sqrt{\sum_{j=i}^N \rho^{6 t_j}}} &
            \frac{\sqrt{\sum_{j=i+1}^N \rho^{4 t_j}}}{\sqrt{\sum_{j=i}^N \rho^{4 t_j}}}
        \end{bmatrix}, \\
        & \hbv_{\ell}^T=\begin{bmatrix}
            \frac{(-1)^{\1(i=N)}}{6}\sqrt{\sum_{j=\ell}^N \rho^{6 t_j}} &
            \frac{\rho^{ t_\ell}}{2}\sqrt{\sum_{j=\ell}^N \rho^{4 t_j}}
        \end{bmatrix},
      \end{aligned}
    \end{equation}
    and the kernel matrix $K_{\boldeta}^{\DC}\in\calS_{N,1}$ with $c=1$ has GvR
    \begin{align}\label{eq:Giv_DC}
        c_i=\frac{(\lambda\rho)^{t_i}}{\sqrt{\sum_{j=i}^N (\lambda \rho)^{2t_j}}}, \ 
        s_i= \frac{\sqrt{\sum_{j=i+1}^{N}(\lambda\rho)^{2t_j}}}{\sqrt{\sum_{j=i}^{N}(\lambda\rho)^{2t_j}}}, \  
        \hv_\ell=  \sbk{\frac{\lambda}{\rho}}^{t_\ell}\sqrt{ \sum_{j=\ell}^{N}(\lambda\rho)^{2 t_j} }, 
    \end{align}
    for $i=1,\ldots,N-1$ and $\ell=1,\ldots,N$.
    Letting $\lambda=\rho$ in the GvR of $K_{\eta}^{\DC}$ gives the GvR of the kernel matrix $K_{\boldeta}^{\TC}\in\calS_{N,1}$. 
\end{proposition} The proof of this proposition is placed in \Cref{sec-proof:Giv_SS_DC_TC} of the Appendix~\ref{sec:proofs}.

In practice, if the GR of a matrix is available, the method given by \cref{eq:gen2giv_G,eq:gen2giv_vhat} provides two ways to construct the GvR of the matrix, either by using \cref{eq:gen2giv_G,eq:gen2giv_vhat} to derive the closed-form expression of $\c_i$, $\s_i$, and $\hbv_{\ell}$ like \cref{eq:Giv_SS,eq:Giv_DC} and their sampled versions \cref{eq:Giv_DC_equspace,eq:Giv_SS_equspace},
or by directly computing its GvR numerically via \cref{eq:gen2giv_G,eq:gen2giv_vhat} when the closed-form expression is hard to derive.
Simulation results in \Cref{sec:experiments} show that even though the GR-based algorithms are numerically unstable, the GvR-based algorithms relying on the above two GvR-construction ways still provide  accurate results.

Constructing GvR via \cref{eq:gen2giv_G,eq:gen2giv_vhat} is also applicable to second-order DC (DC2) and TC (TC2) kernels \cite{Z24}, and other more general kernels, such as the simulation induced (SI) kernels and amplitude modulated locally stationary (AMLS) kernels \cite{CA21}. What's more, if the input $u(t)$ satisfies \cref{eq:input_condition_GR_output_kernel}, then by \cite[Theorem 3]{CA21}, the output kernel $\Psi_{\boldeta}\in\calG_{N,p'+r}\subset\calS_{N,p'+r}$, whose GvR can also be obtained by its GR via \cref{eq:gen2giv_G,eq:gen2giv_vhat}.
For illustration, we show below  the GvR \cref{eq:output_kernel_input_exp_GvR_CT,eq:output_kernel_input_exp_GvR_DT} of $\Psi_{\boldeta}\in\calS_{N,2}$ with the exponential input \cref{eq:input-exp}  
and DC kernel \cref{eq:DC}. 

\begin{proposition}\label{prop:Giv_output}
Consider the output kernel matrix \eqref{eq:output_kernel_mat}.
Suppose that the exponential input \cref{eq:input-exp} and the DC kernel  \cref{eq:DC} are used, and moreover, $T_{\lambda,\rho,\alpha}=\log(\lambda\rho)+\alpha\neq 0$ and $D_{\lambda,\rho,\alpha}=\log(\lambda/\rho)+\alpha\neq 0$.
Then for the CT case, the output kernel matrix $\Phi_{\boldeta}\in\calS_{N,2}$ with the following GvR
\begin{small}
  \begin{equation}\label{eq:output_kernel_input_exp_GvR_CT}
    \begin{aligned}
        \c_i^T &= \begin{bmatrix}
        \frac{\abs{(\lambda\rho)^{t_i}-e^{-\alpha t_i}}}
        {\sqrt{\sum_{j=i}^{N} [(\lambda\rho)^{t_j}-e^{-\alpha t_j}]^2}} &
        \frac{e^{-\alpha t_i}}{\sqrt{\sum_{j=i}^{N}e^{-2\alpha t_j}}}
        \end{bmatrix}, \\
      \s_i^T &= \begin{bmatrix}
        \frac{\sqrt{\sum_{j=i+1}^{N} [(\lambda\rho)^{t_j}-e^{-\alpha t_j}]^2}}
        {\sqrt{\sum_{j=i}^{N} [(\lambda\rho)^{t_j}-e^{-\alpha t_j}]^2}} &
        \frac{\sqrt{\sum_{j=i+1}^{N}e^{-2\alpha t_j}}}{\sqrt{\sum_{j=i}^{N}e^{-2\alpha t_j}}}
      \end{bmatrix}, \\
      \hv_{i,1} &=
        \frac{\abs{(\lambda/\rho)^{t_i}-e^{-\alpha t_i}}
        \sqrt{\sum_{j=i}^{N}[(\lambda\rho)^{t_j}-e^{-\alpha t_j}]^2}}
        {\abs{D_{\lambda,\rho,\alpha}T_{\lambda,\rho,\alpha}}},\\
      \hv_{i,2} &= \frac{(\lambda/\rho)^{t_i}-(\lambda\rho)^{t_i}+C_{\lambda,\rho,\alpha}(\lambda^{2t_i}e^{\alpha t_i}-e^{-\alpha t_i})}{D_{\lambda,\rho,\alpha}T_{\lambda,\rho,\alpha}} \sqrt{\sum_{j=i}^{N}e^{-2\alpha t_j}},
    \end{aligned}
  \end{equation}
\end{small}where $C_{\lambda,\rho,\alpha}={\log\rho}/(\log \lambda + \alpha)$. For the DT case, the output kernel matrix $\Phi_{\boldeta}\in\calS_{N,2}$ with the following GvR
  \begin{small}
  \begin{equation}\label{eq:output_kernel_input_exp_GvR_DT}
    \begin{aligned}
      \c_i^T &= \begin{bmatrix}
        \frac{\abs{e^{-\alpha t_i}-(\lambda\rho)^{t_i}e^{T_{\lambda,\rho,\alpha}}}}
        {\sqrt{\sum_{j=i}^{N} [e^{-\alpha t_j}-(\lambda\rho)^{t_j}e^{T_{\lambda,\rho,\alpha}}]^2}} &
        \frac{e^{-\alpha t_i}}{\sqrt{\sum_{j=i}^{N}e^{-2\alpha t_j}}}
        \end{bmatrix}, \\
        \s_i^T &= \begin{bmatrix}
        \frac{\sqrt{\sum_{j=i+1}^{N} [e^{-\alpha t_j}-(\lambda\rho)^{t_j}e^{T_{\lambda,\rho,\alpha}}]^2}}
        {\sqrt{\sum_{j=i}^{N} [e^{-\alpha t_j}-(\lambda\rho)^{t_j}e^{T_{\lambda,\rho,\alpha}}]^2}} &
        \frac{\sqrt{\sum_{j=i+1}^{N}e^{-2\alpha t_j}}}{\sqrt{\sum_{j=i}^{N}e^{-2\alpha t_j}}}
      \end{bmatrix}, \\\
      \hv_{i,1} &=
        \frac{\abs{e^{-\alpha t_i}-\sbk{{\lambda}/{\rho}}^{t_i}e^{D_{\lambda,\rho,\alpha}}}
        \sqrt{\sum_{j=i}^{N}[e^{-\alpha t_j}-(\lambda\rho)^{t_j}e^{T_{\lambda,\rho,\alpha}}]^2}}
        {\abs{D_{\lambda,\rho,\alpha}'T_{\lambda,\rho,\alpha}'}},\\
      \hv_{i,2} &= \frac{e^{D_{\lambda,\rho,\alpha}}\sbk{{\lambda}/{\rho}}^{t_i}-e^{T_{\lambda,\rho,\alpha}}(\lambda\rho)^{t_i}+C_{\lambda,\rho,\alpha}'(e^{D_{\lambda,\rho,\alpha}+T_{\lambda,\rho,\alpha}}\lambda^{2t_i}e^{\alpha t_i}-e^{-\alpha t_i})}{D_{\lambda,\rho,\alpha}'T_{\lambda,\rho,\alpha}'} \sqrt{\sum_{j=i}^{N}e^{-2\alpha t_j}},
    \end{aligned}
  \end{equation}
  \end{small}where
  $T_{\lambda,\rho,\alpha}'=1-e^{T_{\lambda,\rho,\alpha}}$, 
  $D_{\lambda,\rho,\alpha}'=1-e^{D_{\lambda,\rho,\alpha}}$,
  and 
  $C_{\lambda,\rho,\alpha}'=(e^{D_{\lambda,\rho,\alpha}}-e^{T_{\lambda,\rho,\alpha}})/(1-e^{D_{\lambda,\rho,\alpha}+T_{\lambda,\rho,\alpha}})$.

\end{proposition}
The proof of this proposition is placed in \Cref{sec-proof:Giv_output} of the Appendix~\ref{sec:proofs}.

\section{Algorithms}\label{sec:alg}
In this section, we provide the fast algorithms for computing \cref{eq:sysid_rkhs_alpha} and \cref{eq:EB_opt,eq:SURE_opt,eq:GCV_opt,eq:GML_opt} for KRM.
Specifically, assume that we have the GvR of $\Psi_{\boldeta}\in\calS_{N,p}$, the calculation of \cref{eq:sysid_rkhs_alpha}, $\hat{\y}$, and the criteria \cref{eq:EB_opt,eq:SURE_opt,eq:GCV_opt,eq:GML_opt} are summarized in \Cref{tab:krm_algorithms}.
The computational cost are all  at most $\calO(Np^2)$ flops.
In comparison, for the computation of $\tr(M_{\boldeta,\gamma}^{-1})$, the GR-based implementation in \cite{CA21} costs $\calO(Np^3)$ flops, as the GR representation of $L_{\boldeta,\gamma}^{-1}$ is required.

\begin{table}[htbp]
\centering
\caption{Fast algorithms for computing KRM quantities given $\Psi_{\boldeta}\in\calS_{N,p}$}
\label{tab:krm_algorithms}
\begin{tabular}{lcc}
\hline
Quantitiy & Algorithm & Cost \\
\hline
GvR: $M_{\boldeta,\gamma} = \Psi_{\boldeta} + \gamma I_N$ & -- & -- \\
GvR: Cholesky factor $L_{\boldeta,\gamma}$ of $M_{\boldeta,\gamma}$  & \Cref{alg:Chol} & $\calO(Np^2)$ \\
$\hat{\boldalpha} = L_{\boldeta,\gamma}^{-T}(L_{\boldeta,\gamma}^{-1}\y)$ & \Cref{alg:Lx=y,alg:Ltx=y} & $\calO(Np)$ \\
$\hat{\y} = \Psi_{\boldeta}\hat{\boldalpha}$ & \Cref{alg:Ax} & $\calO(Np)$ \\
$\y^T M_{\boldeta,\gamma}^{-1}\y $ & $\y^T\hat{\boldalpha}$ & $\calO(N)$ \\
$\log\det(M_{\boldeta,\gamma})$ & $2\sum_{i=1}^{N}\log [L_{\boldeta,\gamma}(i,i)]$  & $\calO(N)$ \\
$\tr(M_{\boldeta,\gamma}^{-1})$ &  \Cref{alg:diagADinv} & $\calO(Np^2)$\\
$\tr(H_{\boldeta,\gamma})$ & \Cref{alg:traceADtAtD} & $\calO(Np^2)$ \\
\hline
\end{tabular}
\end{table}

\subsection{Matrix-vector product}\label{sec:Ax}
For the remaining parts of this section, assume $A\in\calS_{N,p}$ and consider the matrix-vector product $A\x$, where $\x\in\R^N$. We decompose 
\begin{align*}
  A\x=\underbrace{\tril(A,-1)\x}_{\y^{\rmL}}+\underbrace{\diag(A)\x}_{\y^{\rmD}}+\underbrace{\triu(A,1)\x}_{\y^{\rmR}},
\end{align*}
and let $\y=\y^{\rmL}+\y^{\rmD}=\tril(A)\x$.
The $i$th element $y_i^\rmL$ of $\y^{\rmL}$ is 
\begin{align}\label{eq:trilAx_yi}
  y_i^\rmL =\c_i^T \boldchi_i, \quad \text{where } \boldchi_i&=\begin{cases}
    \0_p & \text{ if } i=1, \\
    \sum_{j=1}^{i-1}S_{i-1:j}^{>}\hbv_j x_j & \text{ if } i=2,\ldots,N,
  \end{cases}
\end{align}with relation $\boldchi_i=S_{i-1}(\boldchi_{i-1}+\hbv_{i-1}x_{i-1})$ for $i=2,\ldots,N$.
Combined with $y_i^\rmD=(\c_i^T\hbv_i)x_i$, we can write the formulas of $\y$ as the so-called \textit{discrete-time forward system with homogeneous boundary conditions} (DTFSwHBC) \cite{eidelman2013separable}:
\begin{equation}\label{eq:trilAx_forwardSys}
  \begin{cases}
    \boldchi_i=S_{i-1}\boldchi_{i-1}+S_{i-1}\hbv_{i-1}x_{i-1}, & i=2,\ldots,N\\
    y_i = \c_i^T \boldchi_i + (\c_i^T\hbv_i)x_i, & i=1,\ldots,N\\
    \boldchi_1=\0_p,
  \end{cases} 
\end{equation}
where $x_i$, $y_i$, and $\boldchi_i$ are called the system input, output, and state, respectively. Often $\tril(A)$ is called the matrix of the input-output operator
of the system \cref{eq:trilAx_forwardSys}. 
The algorithm to compute $\y=\tril(A)\x$ in terms of system \cref{eq:trilAx_forwardSys} was firstly introduced by \cite{eidelman2013separable} and clearly, $\y^{\rmR}$ can be computed in a similar way.
\Cref{alg:Ax} shows the recursive evaluation of $A\x$ in $\calO(Np)$ flops, the same order as the GR-based implementation \cite{Vandebril2008a,andersen2020smoothing}.
The high-efficiency comes from the small sub-block rank and column/row dependency. 
For example, for $p=1$ and $i=1,\ldots,N-1$, we have $(s_i c_{i+1}/c_i)\cdot\tril(A)(i,1:i)=\tril(A)(i+1,1:i)$, thus allowing a recursive relation. 

\begin{algorithm}
  \caption{Matrix-vector product $A\x$}
  \label{alg:Ax}
  \begin{algorithmic}
  \STATE{\textbf{Input:} GvR $\c_i,\s_i,\hbv_i\in\R^p$ of $A\in\calS_{N,p}$, and $\x\in\R^N$.}  
  \STATE{\textbf{Output:} $\z\in\R^N$ such that $A\x=\z$.} 
  \STATE{Initialize $\boldchi^\rmL\leftarrow\0_p$; $\boldchi^\rmR\leftarrow\0_p$;} 
  \FOR{$i=1\ldots,N$}
    \STATE{$y_i^\rmL\leftarrow \c_i^T\boldchi^\rmL$;}
    \STATE{$\boldchi^\rmL\leftarrow \s_{i}\circ(\boldchi^\rmL+\hbv_{i}x_{i})$ if $i\neq N$;}
    \STATE{$y_i^\rmD\leftarrow \c_i^T\hbv_i x_i$;}
  \ENDFOR
  \FOR{$i=N,\ldots,1$}
    \STATE{$y_i^\rmR\leftarrow \hbv_i^T\boldchi^\rmR$;}
    \STATE{$\boldchi^\rmR\leftarrow \s_{i-1}\circ(\boldchi^\rmR+\c_{i}x_{i})$ if $i\neq 1$;}
  \ENDFOR
  \STATE{$\z\leftarrow\y^\rmL+\y^\rmD+\y^\rmR$}
  \end{algorithmic}
\end{algorithm}

\subsection{Cholesky factorization of $A+D$}\label{sec:Chol}
Let $\d\in\R_+^N$ and $D=\diag(\d)$, then the $(i,j)$-entry of $A+D$ can obviously be written as follows:
\begin{align}\label{eq:giv_mat_quasi}
  (A+D)(i,j) = 
  \begin{cases}
    \c_i^T S_{i-1:j}^{>} \hbv_j &\text{ if } 1\leq j< i\leq N,\\
    \c_i^T \hbv_i + d_i &\text{ if } 1\leq i=j\leq N,\\
    \c_j^T S_{j-1:i}^{>} \hbv_i &\text{ if } 1\leq i<j\leq N, 
  \end{cases} 
\end{align}
As well known, the semiseparable-plus-diagonal matrix $A+D$ belongs to the class of the \textit{symmetric quasiseparable} or \textit{$p$-quasiseparable matrices} with quasiseparability rank $p$ \cite{Vandebril2008a}. 
Below, it is shown that the Cholesky factor $L$ of $(A+D)=LL^T$ has a representation containing the Givens-vector $\c_i$ and $\s_i$.
\begin{proposition}\label{prop:L}
    Suppose $A+D\succ \0$, then the Cholesky factor $L$ of $A+D$ has $(i,j)$-entry
    \begin{align}\label{eq:chol_L}
    L (i,j) = 
    \begin{cases}
    \c_i^T S_{i-1:j}^{>} \w_j &\text{ if } 1\leq j< i\leq N,\\
    f_i &\text{ if } 1\leq i=j\leq N,
  \end{cases} 
\end{align}where $f_i>0$ and $\w_i\in\R^p$ have the recursive relations 
\begin{align*}
    f_i &= \sqrt{\c_i^T(\hbv_i-P_i\c_i)+d_i}, \quad i=1,\ldots,N, \\
    \w_i &= (\hbv_i-P_i\c_i)/f_i, \qquad \quad \hspace{9pt} i=1,\ldots,N-1,
\end{align*}with $P_1=\0_{p\times p}$ and $P_i=S_{i-1}(\w_{i-1}\w_{i-1}^T+P_{i-1})S_{i-1}$ for $i=2,\ldots,N$.
\end{proposition}
\begin{proof}

For $i=2,\ldots,N-1$, we introduce the following block partitions:
\begin{displaymath}
  A+D = \begin{bmatrix}
    A_{11} + D_{11} & A_{21}^T \\ A_{21} & A_{22}+D_{22}
  \end{bmatrix}, \quad
  L = \begin{bmatrix}
    L_{11} & \0 \\ L_{21} & L_{22}
  \end{bmatrix},
\end{displaymath}
where $A_{11}=A(1:i-1,1:i-1)$, $A_{21}=A(i:N,1:i-1)$, $A_{22}=A(i:N,i:N)$, $D_{11}=\diag(d_1,\ldots,d_{i-1})$, and $D_{22}$, $L_{11}$, $L_{21}$, and $L_{22}$ are similar decomposed blocks.
\begin{displaymath}
  L_{21}=\begin{bmatrix}
    \c_i^T S_{i-1:1}^>\w_1 & \cdots & \c_i^T S_{i-1}\w_{i-1} \\
    \vdots & \ddots & \vdots \\
    \c_N^T S_{N-1:1}^>\w_1 & \cdots & \c_N^T S_{N-1:i-1}^>\w_{i-1}
  \end{bmatrix}, \
  L_{22}=\begin{bmatrix}
    f_i & \0 \\
    \tilde{L}_{22} & \cdots
  \end{bmatrix}, 
\end{displaymath}
where $\tilde{L}_{22}=L_{22}(2:N-i+1,1)$. 
By $L_{22}L_{22}^T = A_{22} + D_{22} - L_{21}L_{21}^T$, it follows from the $(1,1)$-entry of $L_{22}L_{22}^T$ that
\begin{align*}
  f_i^2 &= \c_i^T \hbv_i + d_i - \sum_{j=1}^{i-1}\c_i^T S_{i-1:j}^>\w_j\w_j^T S_{j:i-1}^<\c_i 
  = \c_i^T \sbk{\hbv_i - P_i\c_i} + d_i, \\
  P_i &= \sum_{j=1}^{i-1} S_{i-1:j}^>\w_j\w_j^T S_{j:i-1}^<.
\end{align*}
Next, the remaining rows of the first column of $L_{22}L_{22}^T$ is 
\begin{align*}
  &L_{22}L_{22}^T(2:N-i+1,1) = f_i \tilde{L}_{22}\\
  =&\begin{bmatrix}
    \c_{i+1}^T S_{i}\hbv_i \\ 
    \c_{i+2}^T S_{i+1:i}^>\hbv_i \\
    \vdots \\
    \c_{N}^TS_{N-1:i}^>\hbv_i
  \end{bmatrix}-
  \begin{bmatrix}
    \c_{i+1}^TS_i\sum_{j=1}^{i-1} S_{i-1:j}^> \w_j \w_j^T S_{j:i-1}^< \c_i \\
    \c_{i+2}^TS_{i+1:i}^>\sum_{j=1}^{i-1} S_{i-1:j}^> \w_j \w_j^T S_{j:i-1}^< \c_i \\
    \vdots \\
    \c_{N}^TS_{N-1:i}^>\sum_{j=1}^{i-1} S_{i-1:j}^> \w_j \w_j^T S_{j:i-1}^< \c_i 
  \end{bmatrix}\\ 
  =&\begin{bmatrix}
    \c_{i+1}^T S_i (\hbv_i - P_i \c_i) & \cdots & 
    \c_{N}^T S_{N-1:i}^> (\hbv_i - P_i \c_i)
  \end{bmatrix}^T,
\end{align*} 
hence we obtain
\begin{align*}
  \w_i &= \frac{\hbv_i - P_i \c_i}{f_i}, \quad  \tilde{L}_{22}=\begin{bmatrix}
  \c_{i+1}S_i \w_i & \cdots & \c_N^T S_{N-1:i}^>\w_i
  \end{bmatrix}^T.
\end{align*}
Thus, we have shown the recursive relations for $f_i$ and $\w_i$ by defining $P_1=\0_{p\times p}$ and $P_i=S_{i-1}(\w_{i-1}\w_{i-1}^T+P_{i-1})S_{i-1}$ for $i=2,\ldots,N$.
\end{proof}

\Cref{alg:Chol} computes all the $f_i$ and $\w_i$ in $\calO(Np^2)$ flops, which obtains the same order of complexity as GR.  
Note that if $\d=\0_p$ and $A\succ \0$, then by \Cref{alg:Chol} with $d_i=0$, the Cholesky factor $L$ inherit the semiseparability structure with $L(i,j)=\c_i^T S_{i-1:j}^> \w_j$ for $1\leq j\leq i\leq N$.
To see this, let $\tilde{\w}_i=\hbv_i-P\c_i$, then 
\begin{align*}
  \c_i^T\w_i=\frac{\tilde{\w}_i}{f_i} = \frac{\c_i^T\tilde{\w}_i}{(\c_i^T\tilde{\w}_i+d_i)^{1/2}} \overset{d_i=0}{=} (\c_i^T\tilde{\w}_i)^{1/2} = (\c_i^T\tilde{\w}_i+0)^{1/2} = f_i = L(i,i).
\end{align*}

\begin{algorithm}
  \caption{Cholesky factorization of $A+D=LL^T$, where $A+D \succ \0$.}
  \label{alg:Chol}
  \begin{algorithmic}
  \STATE{\textbf{Input:} GvR $\c_i,\s_i,\hbv_i\in\R^p$ of $A\in\calS_{N,p}$ and $\d\in\R_+^N$ such that $A+D\succ \0$.}  
  \STATE{\textbf{Output:} $\w_i\in\R^p$ ($i=1,\ldots,N-1$)  and $f_i$ ($i=1,\ldots,N$) in \cref{eq:chol_L}.} 
  \STATE{Initialize $P\leftarrow \0_{p\times p}$;}
  \FOR{$i=1,\ldots,N$}
    \STATE{$\w_i\leftarrow\hbv_i - P \c_i$; \quad $f_i\leftarrow (\c_i^T\w_i+d_i)^{1/2}$;}
    \STATE{$\w_i\leftarrow\w_i/f_i$;}
    \STATE{$P\leftarrow S_{i}(\w_i\w_i^T + P)S_{i}$ if $i\neq N$;}
  \ENDFOR
  \end{algorithmic}
\end{algorithm}

The representation \cref{eq:chol_L} allows us to compute the determinant of $(A+D)$ by $\det(A+D)=\det(L L^T)=\prod_{i=1}^{n} f_i^2$,
and the products $L\x$ and $L^T\x$ via \Cref{alg:Lx,alg:Ltx} in $\calO(Np)$ flops, which are served as a special case of \Cref{alg:Ax} with $\c_i\hbv_i$ replaced by $f_i$.
Hence, 
the $i$th element $y_i$ of product $L\x=\y$ is
\begin{align*}
  \c_i^T \boldchi_i + f_i x_i= y_i, \quad \text{where } \boldchi_i=\begin{cases}
    \0_p & \text{ if } i=1, \\
    \sum_{j=1}^{i-1}S_{i-1:j}^{>}\w_j x_j & \text{ if } i=2,\ldots,N,
  \end{cases}
\end{align*}
and similar to \cref{eq:trilAx_forwardSys}, we can write
\begin{equation}\label{eq:Lx=y_system}
  \begin{cases}
    \boldchi_i=S_{i-1}\boldchi_{i-1}+S_{i-1}\w_{i-1}x_{i-1}, & i=2,\ldots,N\\
    y_i = \c_i^T \boldchi_i + f_i x_i, & i=1,\ldots,N,\\
    \boldchi_1=\0_p.
  \end{cases} 
\end{equation}
We can also compute the forward/backward substitution $L\x=\y$ and $L^T \x=\y$ for $\x,\y\in\R^N$.
For example, for the forward substitution, writing the second line of \cref{eq:Lx=y_system} as $x_i=f_i^{-1}(y_i-\c_i^T\boldchi_i)$ gives the solution $\x$ recursively.
Hence the recursive implementation to compute $L\x=\y$ and $L^T \x=\y$ for $\x,\y\in\R^N$ cost only $O(Np)$ flops as well. See \Cref{alg:Lx=y,alg:Ltx=y}.

\subsection{Inverse of Cholesky factor}\label{sec:Chol_inv}

Assume the same settings as \Cref{sec:Chol} and $\d\in\R_{++}^N$.
Since the Cholesky factor $L$ in \cref{eq:chol_L} is $p$-quasiseparable, $L^{-1}$ is $p$-quasiseparable as well by \cite[Theorem 8.46]{Vandebril2008a}.
To calculate $L^{-1}$, first recall from \Cref{sec:Ax} that we can compute the product $L\x=\y$ through its associated DTFSwHBC \cref{eq:Lx=y_system}  with state $\boldchi_i$, input $\x$, and output $\y$. Then it is interesting to note that \cite{eidelman2013separable} proposes a method for computing the inverse of (block) quasiseparable matrix $\tilde{L}\in\R^{N\times N}$ by using the DTFSwHBC associated with $\tilde{L}\x=\y$. 
Specifically,
for a lower-triangular matrix $\tilde{L}\in\R^{N\times N}$ with block sizes $1\times 1$ for simplicity and quasiseparable generators $\tilde{L}(i,j)=\p_i^T A_{i:j}^> \q_j$ for $1\leq j < i \leq N$, and $\tilde{L}(i,i)=g_i$ for $i=1,\ldots,N$,
then by \cite[Theorem 13.2,Theorem 13.3,Corollary 13.5]{eidelman2013separable}, its corresponding DTFSwHBC for $\tilde{L}\x=\y$ with state $\tilde{\boldchi}_i$, input $\x$, and output $\y$ is 
\begin{align}\label{eq:Lx=y_quasi_system}
  \begin{cases}
    \tilde{\boldchi}_i=A_{i-1}\tilde{\boldchi}_{i-1}+\q_{i-1}x_{i-1}, & i=2,\ldots,N,\\
    y_i = \p_i^T \tilde{\boldchi}_i + g_i x_i, & i=1,\ldots,N,\\
    \tilde{\boldchi}_1=\0_p,
  \end{cases} 
\end{align}
where the coefficients $(A_i,\q_i,\p_i,g_i)$ are dependent on the quasiseparable generators of $\tilde{L}$. 
The idea to compute $\tilde{L}^{-1}$ is that by swapping the input $\x$ and output $\y$ of \eqref{eq:Lx=y_quasi_system}, we obtain a new DTFSwHBC, which can be shown to be the one associated with $\tilde{L}^{-1}\y=\x$. 
Moreover, we can derive from the coefficients of this new DTFSwHBC the corresponding quasiseparable generators of $\tilde{L}^{-1}$.

In the following, we extend this idea from quasiseparable matrices to semiseparable matrices. 
First, we start from the DTFSwHBC \cref{eq:Lx=y_system}, swap 
its input $\x$ and output $\y$, and obtain for $i=2,\ldots,N$,
\begin{align*}
    \boldchi_i 
    &= S_{i-1} (I_p-\w_{i-1}f_{i-1}^{-1}\c_{i-1}^T)[\boldchi_{i-1}+(I_p-\w_{i-1}f_{i-1}^{-1}\c_{i-1}^T)^{-1}\w_{i-1}f_{i-1}^{-1}y_{i-1}].
\end{align*}Then we let $\bar{S}_i=S_{i} (I_p-\w_{i}f_{i}^{-1}\c_{i}^T)$ and $\bar{\w}_i= f_i^{-1}(I_p-\w_i f_i^{-1}\c_i^T)^{-1}\w_i$, and obtain a new DTFSwHBC as follows  
\begin{align}\label{eq:Lx=y_inv_system}
    \begin{cases}
        \boldchi_{i}=\bar{S}_{i-1}\boldchi_{i-1} + \bar{S}_{i-1}\bar{\w}_{i-1}y_{i-1}, & i=2,\ldots,N\\
        x_i=(-f_i^{-1}\c_i^T)\boldchi_i + f_i^{-1} y_i, & i=1,\ldots,N\\
        \boldchi_1=\0_p.
    \end{cases} 
\end{align}   
Furthermore, we  ``guess'' that the coefficients $(\bar{S}_{i-1},\bar{S}_{i-1}\bar{\w}_{i-1},-f_i^{-1}\c_i^T,f_i^{-1})$ of the DTFSwHBC \eqref{eq:Lx=y_inv_system} can be used to construct an implicit representation of $L^{-1}$ and prove it rigorously in the following  \Cref{thm:Chol_inv}.
\begin{remark}
    This extension is non-trivial.
    In \cref{eq:Lx=y_quasi_system}, the coefficients $(A_i,\q_i)$ are directly obtained form the quasiseparable generators of $\tilde{L}$, which is not the case for our coefficients $(S_{i-1},S_{i-1}\w_{i-1})$ in \cref{eq:Lx=y_system} due to the common term $S_{i-1}$.
    This feature also makes the the coefficients $(\bar{S}_{i-1},\bar{S}_{i-1}\bar{\w}_{i-1})$  in DTFSwHBC \cref{eq:Lx=y_inv_system} more complicated as $\bar{\w}_{i}$ involves the matrix inverse $(I_p-\w_if_i^{-1}\c_i^T)^{-1}$, whose existence should be guaranteed. 
\end{remark}
\begin{theorem}[Inverse of $L$]\label{thm:Chol_inv}
  Suppose $A+D\succ\0$. 
  Let $L$ in \cref{eq:chol_L} be the Cholesky factor of $A+D=LL^T$. If $f_i>0$ and $d_i > 0$ for all $i=1,\ldots,N$, then $L^{-1}$ can be represented by 
  \begin{align}\label{eq:chol_Linv}
  L^{-1}(i,j) = 
  \begin{cases}
    \bar{\c}_i^T \bar{S}_{i-1:j}^{>} \bar{\w}_j &\text{ if } 1\leq j< i\leq N,\\
    \bar{f}_i &\text{ if } 1\leq i=j\leq N,
  \end{cases}
  \end{align}
where 
\begin{equation}\label{eq:inv_rep}
    \begin{aligned}
        \bar{\c}_i &= -f_i^{-1}\c_i, \qquad \qquad \qquad \qquad \ 
        \bar{S}_i = {S}_i(I_p-\w_i f_i^{-1} \c_i^T), \\
        \bar{\w}_i&= f_i^{-1}(I_p-\w_i f_i^{-1}\c_i^T)^{-1}\w_i,\quad 
        \bar{f}_i = f_i^{-1},
    \end{aligned} 
\end{equation}
with $\bar{\c}_i$ and $\bar{f}_i$ range from $i=1,\ldots,N$, and $\bar{S}_i$ and $\bar{\w}_i$ range from $i=1,\ldots,N-1$.
\end{theorem}

\begin{proof}
    The conditions $f_i>0$ and $d_i>0$ guarantee $f_i-\c_i^T\w_i>0$ and the existence of $(I_p-\w_i f_i^{-1}\c_i^T)^{-1}$. 
    To see this, recall \Cref{alg:Chol} that if $f_i=(\c_i^T\tilde{\w}_i+d_i)^{1/2}$ and $\w_i=\tilde{\w}_i/f_i$ for some $\tilde{\w}_i\in\R^p$, then 
    \begin{align}
      f_i-\c_i^T\w_i = \frac{f_i^2-\c_i^T\tilde{\w}_i}{f_i}=\frac{\c_i^T\tilde{\w}_i+d_i-\c_i^T\tilde{\w}_i}{f_i}=\frac{d_i}{f_i} > 0. \label{eq:dfratio}
    \end{align}
    Moreover, by matrix inversion lemma, in $\bar{\w}_i$,
    \begin{align}\label{eq:matrix_inverse_w}
        (I_p-\w_if_i^{-1}\c_i^T)^{-1} = I_p + \w_i(f_i-\c_i^T \w_i)^{-1}\c_i^T,
    \end{align}
    so the condition $f_i-\c_i^T\w_i>0$ implies the existence of $(I_p-\w_if_i^{-1}\c_i^T)^{-1}$. 
 
    Now we show that $L$ in \cref{eq:chol_L} and $L^{-1}$ in \cref{eq:chol_Linv} satisfy $L^{-1}L=I_N$.
    It is equivalent to show for any $\x\in\R^N$, $L^{-1}L\x=\x$ with $\y:=L\x$ satisfying \cref{eq:Lx=y_system}.
    We show it by computing $(L^{-1}\y)_i$ for each $i=1,\ldots,N$.
    In the proof, we use the important relation 
    $\bar{S}_i \bar{\w}_i = f_i^{-1} S_i \w_i$ for $i=1,\ldots,N$.
    Recall \cref{eq:Lx=y_system} that $\boldchi_i= \sum_{j=1}^{i-1}S_{i-1:j}^> \w_j x_j
    =S_{i-1}\w_{i-1}x_{i-1}+S_{i-1}\boldchi_{i-1}$ for $i>1$ and $y_i=\c_i^T \boldchi_i + f_i x_i$, we first show
    \begin{align}
        \sum_{k=1}^i \bar{S}_{i:k}^> \bar{\w}_k y_k = \sum_{k=1}^i {S}_{i:k}^> {\w}_k x_k = \boldchi_{i+1}, \quad i=1,\ldots,N-1, \label{eq:proof_inv_chi}
    \end{align}
    by induction.
    When $i=1$, the left hand side $\bar{S}_1\bar{\w}_1 y_1
    =S_1\w_1x_1=\boldchi_2$.
    Suppose \cref{eq:proof_inv_chi} holds for some $i\leq N-2$, then we show \cref{eq:proof_inv_chi} also holds for $i+1$.
    Note that
    \begin{align*}
        \sum_{k=1}^{i+1} \bar{S}_{i+1:k}^> \bar{\w}_k y_k &= \bar{S}_{i+1} \bar{\w}_{i+1}y_{i+1} + \bar{S}_{i+1} \sum_{k=1}^{i} \bar{S}_{i:k}^> \bar{\w}_k y_k 
        =\bar{S}_{i+1} \bar{\w}_{i+1}y_{i+1} + \bar{S}_{i+1} \boldchi_{i+1}
    \end{align*}
    by induction hypothesis. 
    By $\boldchi_{i+2}=S_{i+1}\w_{i+1}x_{i+1}+S_{i+1}\boldchi_{i+1}$ and $y_{i+1}=\c_{i+1}^T\boldchi_{i+1}+f_{i+1}x_{i+1}$, we have
    \begin{align*}
        & \sum_{k=1}^{i+1} \bar{S}_{i+1:k}^> \bar{\w}_k y_k-\boldchi_{i+2} = \bar{S}_{i+1} \bar{\w}_{i+1}y_{i+1} - S_{i+1}\w_{i+1}x_{i+1} + (\bar{S}_{i+1}-S_{i+1}) \boldchi_{i+1} \\
        & \quad = f_{i+1}^{-1}S_{i+1}\w_{i+1}(\c_{i+1}^T\boldchi_{i+1}+f_{i+1}x_{i+1}) - S_{i+1}\w_{i+1}x_{i+1} - S_{i+1}\w_{i+1}f_{i+1}^{-1}\c_{i+1}^T \boldchi_{i+1},
    \end{align*}which equals to $\0_p$.
    Therefore, by induction, we have shown \cref{eq:proof_inv_chi}.

    Then we prove $(L^{-1}\y)_i=x_i$ for all $i=1,\ldots,N$.
    When $i=1$, $(L^{-1}\y)_1 = \bar{f}_1(\c_1^T\boldchi_1 + f_1 x_1)=\bar{f}_1 f_1 x_1$ by $\boldchi_1=\0_p$.
    For $i>1$, 
    \begin{align*}
        (L^{-1}\y)_{i} &= \sum_{k=1}^{i}[L^{-1}(i,k)] y_k 
        = f_i^{-1} y_i - f_i^{-1} \c_i^T \sum_{k=1}^{i-1}\bar{S}_{i-1:k}^> \bar{\w}_k y_k 
        \overset{\cref{eq:proof_inv_chi}}{=} \frac{1}{f_{i}} \sbk{y_{i} - \c_i^T \boldchi_{i}}
    \end{align*}which is $x_i$.
    This completes the proof.
\end{proof}

As shown by \cref{eq:dfratio,eq:matrix_inverse_w}, the two conditions $f_i>0$ and $d_i>0$ guarantee the existence of the inverse $(I_p-\w_if_i^{-1}\c_i^T)^{-1}$.
They are mild, since $f_i>0$ by positive definiteness, and in KRM, $D=\gamma I_N$ for some $\gamma>0$ such that $d_i=\gamma>0$. 
On the other hand, if $d_i=0$ for all $i$, the Cholesky factor $L$ is $p$-semiseparable and cannot attain the form \cref{eq:chol_Linv}. 
By \cite[Theorem 8.45]{Vandebril2008a}, $L^{-1}$ is a lower-triangular band matrix with bandwidth $p$.
The details are skipped here.

\begin{remark}\label{rmk:dpp_gen}
  The condition $\d\in\R_{++}^N$ also appears in the GR-based algorithm.
  Suppose $A\in\calG_{N,p}$ has GR \cref{eq:gen-p-sym-semi} and $\d\in\R_{++}^N$, then the Cholesky factor $L=\tril(U Q^T,-1)+ \diag(\g)$ for some $\g\in\R_{++}^N$. 
  \cite[Theorem 4.1]{andersen2020smoothing} shows that $L^{-1}=\tril(YZ^T,-1)+\diag(\g)^{-1}$, whose existence relies on the non-singularity of $QY^T-I_N$, which guaranteed by  $g_i-\u_i^T\q_i=d_i/g_i>0$, analogous to \cref{eq:dfratio}. 
  Here $\u_i$ and $\q_i$ are the $i$th row of $U$ and $Q$, respectively. 
\end{remark}
Obviously,  computing $L^{-1}$ through its implicit representation in \cref{eq:inv_rep}  requires $\calO(Np^2)$ flops, as it involves the inverse $(I_p-\w_if_i^{-1}\c_i^T)^{-1}$ obtained through \cref{eq:matrix_inverse_w}, which may be inaccurate when $f_i$ is close to $\c_i^T\w_i$.
However, it is worth mentioning that, 
we find a different route to compute the trace of $L^{-1}$ with $\calO(Np^2)$ flops and without using the implicit representation of $L^{-1}$ in \cref{eq:inv_rep}, as detailed in the following section.

\subsection{Algorithms concerning the trace of inversion}\label{sec:alg-tr-inv}
Assume the same settings as \Cref{sec:Chol} and $\d\in\R_{++}^N$.
\subsubsection{The diagonal elements of $(A+D)^{-1}$}
Let $\b\in\R^N$ with $b_i$ the $i$th diagonal element of $(A+D)^{-1}$, i.e., $b_i=\e_i^T(A+D)^{-1}\e_i=\norm{L^{-1}\e_i}_2^2$, where $L$ is the Cholesky factor in \cref{eq:chol_L}. By the representation \cref{eq:inv_rep} of $L^{-1}$ in \Cref{thm:Chol_inv}, we have
\begin{align*}
  b_i&=f_i^{-2} + \sum_{j=i+1}^{N}[\bar{\w}_i^T (\bar{S}_{j-1:i}^>)^T\bar{\c}_j][\bar{\c}_j^T \bar{S}_{j-1:i}^> \bar{\w}_i] 
  = f_i^{-2} + f_i^{-2}{\w}_i^T {S}_{i}^TP_i {S}_i{\w}_i
\end{align*} 
where $P_i=
  \sum_{j=i+1}^{N}(\bar{S}_{j-1:i+1}^>)^T\bar{\c}_j\bar{\c}_j^T\bar{S}_{j-1:i+1}^>
\in\R^{p\times p}$ for $i\leq N-1$ and $\0_{p\times p}$ for $i=N$.
Define $R_i=S_i^T P_i S_i$ and $\p_i=R_i \w_i$, then $b_i=f_i^{-2}(1+\w_i^T\p_i)$, and  
\begin{align*}
  P_i&=\bar{\c}_{i+1}\bar{\c}_{i+1}^T + \bar{S}_{i+1}^T P_{i+1} \bar{S}_{i+1} 
=b_{i+1}\c_{i+1}\c_{i+1}^T - f_{i+1}^{-1}(\c_{i+1}\p_{i+1}^T + \p_{i+1}\c_{i+1}^T) + R_{i+1}.
\end{align*}   
\Cref{alg:diagADinv} computes $\b$ recursively in $\calO(Np^2)$ flops.

\begin{algorithm}
  \caption{Diagonal elements of $(A+D)^{-1}$, where $A+D \succ 0$ and $\d\in\R_{++}^N$.}
  \label{alg:diagADinv}
  \begin{algorithmic}
  \STATE{\textbf{Input:} Representation $\c_i$, $\s_i$, $\w_i$ and $\f$ of $L$ in \cref{eq:chol_L};}  
  \STATE{\textbf{Output:} $\b\in\R^{N}$ such that $b_i=\e_i^T(A+D)^{-1}\e_i=\norm{L^{-1}\e_i}_2^2$;} 
  \STATE{Initialize $P\leftarrow \0_{p\times p}$, $R\leftarrow \0_{p\times p}$, $\p\leftarrow\0_p$;}
  \STATE{$b_N\leftarrow f_N^{-2}$;} 
  \FOR{$i=N-1,\ldots,1$}
    \STATE{$P\leftarrow b_{i+1}\c_{i+1}\c_{i+1}^T - f_{i+1}^{-1}(\c_{i+1}\p^T + \p\c_{i+1}^T) + R$;}
    \STATE{$R \leftarrow S_i^T P S_i$;\quad $\p\leftarrow R \w_i$; \quad $b_i\leftarrow f_i^{-2}(1+\w_i^T\p)$;}
  \ENDFOR
  \end{algorithmic}
\end{algorithm}

\subsubsection{The trace of $(A+D)^{-1}(\tilde{A}+\tilde{D})$}
Let $\tilde{\d}\in\R^N$, $\tilde{D}=\diag(\tilde{\d})$, and $\tilde{A}\in\calS_{N,\tilde{p}}$ with GvR $\tilde{c}_{i,k}, \tilde{s}_{i,k}$ and $\tilde{\nu}_{i,k}$ for $i=1,\ldots,N$, $k=1,\ldots,\tilde{p}$ such that          
\begin{align}\label{eq:tildeA}
  \tilde{A}(i,j) = 
  \begin{cases}
    \tilde{\c}_i^T \tilde{S}_{i-1:j}^{>} \tbv_j &\text{ if } 1\leq j\leq i\leq N,\\
    \tilde{\c}_j^T \tilde{S}_{j-1:i}^{>} \tbv_i &\text{ if } 1\leq i<j\leq N,  
  \end{cases} 
\end{align}
where $\tilde{\c}_{\bullet}$, $\tilde{S}_{\bullet}$ and $\tbv_{\bullet}$ have the same form as the ones in \cref{eq:giv_mat_p}. 
By \cref{eq:chol_Linv}, 
\begin{align*}
  \tr((A+D)^{-1}(\tilde{A}+\tilde{D}))&=\tr(L^{-1}(\tilde{A}+\tilde{D})L^{-T})=\sum_{i=1}^{N} q_i,
\end{align*}where $q_i:=\e_i^T L^{-1} (\tilde{A}+\tilde{D}) L^{-T}\e_i$.
Denote $\tilde{A}_{i}+\tilde{D}_{i}$ the leading principal minor of $\tilde{A}+\tilde{D}$ of order $i$, and $\tilde{V}_{i}:=\begin{bmatrix}
        \tilde{S}_{i:1}^>\tbv_1 & \tilde{S}_{i:2}^>\tbv_2 & \cdots & \tilde{S}_{i}\tbv_{i}
    \end{bmatrix}\in\R^{\tilde{p}\times i}$ and $\bar{W}_i:=\begin{bmatrix}
        \bar{S}_{i:1}^>\bar{\w}_1 & \bar{S}_{i:2}^>\bar{\w}_2 & \cdots & \bar{S}_{i}\bar{\w}_{i}
    \end{bmatrix}\in\R^{p\times i}$.
Since $L^{-T}\e_i=\begin{bmatrix}
    \bar{\c}_i^T\bar{W}_{i-1} & \bar{f}_i & \0_{N-i}^T
\end{bmatrix}^T$, $q_i$ only contains the first $i$th element of $L^{-T}\e_i$,
\begin{align*}
    q_i &= \begin{bmatrix}
    \bar{\c}_i^T\bar{W}_{i-1} & \bar{f}_i 
\end{bmatrix} \begin{bmatrix}
    \tilde{A}_{i-1} + \tilde{D}_{i-1} & \tilde{V}_{i-1}^T\tilde{\c}_i \\
    \tilde{\c}_i^T\tilde{V}_{i-1} & \tilde{\c}_i^T\tbv_i+\tilde{d}_i
\end{bmatrix} \begin{bmatrix}
    \bar{W}_{i-1}^T \bar{\c}_i \\ \bar{f}_i
\end{bmatrix} \\
    &=\bar{\c}_i^T \bar{W}_{i-1} (\tilde{A}_{i-1}+\tilde{D}_{i-1})\bar{W}_{i-1}^T \bar{\c}_i + 2\bar{f}_i \tilde{\c}_i^T \tilde{V}_{i-1} \bar{W}_{i-1}^T \bar{\c}_i + \bar{f}_i^2(\tilde{\c}_i^T\tbv_i+\tilde{d}_i), \ i\geq 2,
\end{align*}
and $q_1={f}_1^{-2}(\tilde{\c}_1^T\tbv_1+\tilde{d}_1)$. 
For simplicity, we define $R_i:=\tilde{V}_i\bar{W}_i^T\in\R^{\tilde{p}\times p}$ and symmetric $P_i=\bar{W}_i(\tilde{A}_i+\tilde{D}_i)\bar{W}_i^T\in\R^{p\times p}$ with $R_0=\0_{\tilde{p}\times p}$ and $P_0=\0_{p\times p}$. 
We further define vectors $\p_i=P_{i-1}\c_i\in\R^{p}$ and $\r_i=R_{i-1}^T\tilde{\c}_i\in\R^{p}$. Then $q_i$ can be rewritten as 
\begin{align*}
  q_i=f_i^{-2}(\c_i^T\p_i-2\r_i^T\c_i+\tilde{\c}_i^T\tbv_i+\tilde{d}_i),
\end{align*}
where by \cref{eq:inv_rep}, $R_i$ and $P_i$ have recursive relations  
\begin{align*}
  R_i&=\sum_{j=1}^{i}\tilde{S}_{i:j}^>\tbv_j\bar{\w}_j^T (\bar{S}_{i:j}^>)^T 
  =\tilde{S}_{i}\tbv_i\bar{\w}_i^T\bar{S}_{i}^T + \tilde{S}_iR_{i-1}\bar{S}_i^T\\
  &=\tilde{S}_i\mbk{R_{i-1}+f_i^{-1}(\tbv_i-R_{i-1}\c_i)\w_i^T}S_i,\\
  P_i&=\begin{bmatrix}
    \bar{S}_i\bar{W}_{i-1} & \bar{S}_i\bar{\w}_i 
  \end{bmatrix}\begin{bmatrix}
    \tilde{A}_{i-1} + \tilde{D}_{i-1} & \tilde{V}_{i-1}^T\tilde{\c}_i \\
    \tilde{\c}_i^T\tilde{V}_{i-1} & \tilde{\c}_i^T\tbv_i+\tilde{d}_i
\end{bmatrix}\begin{bmatrix}
  \bar{W}_{i-1}^T\bar{S}_i^T \\ \bar{\w}_i^T\bar{S}_i^T
\end{bmatrix} \\
&=\bar{S}_i\mbk{P_{i-1}+R_{i-1}^T\tilde{\c}_i\bar{\w}_i^T + \bar{\w}_i\tilde{\c}_i^T R_{i-1} + \bar{\w}_i(\tilde{\c}_i^T\tbv_i+\tilde{d}_i)\bar{\w}_i^T}\bar{S}_i^T\\
&=S_i\cbk{
  P_{i-1} + f_i^{-1}\mbk{(\r_i-\p_i)\w_i^T + \w_i(\r_i^T-\p_i^T)} + q_i\w_i\w_i^T}S_i,
\end{align*}
for $i=1,\ldots,N$.
\Cref{alg:traceADtAtD} implements this calculation in $\calO (Np\tilde{p})$ flops. Algorithm \ref{alg:diagADinv} is a special case of this algorithm by letting $\tilde{A}=\0_{N\times N}$ and $\tilde{D}=I_N$, i.e., all the $\tilde{c}_i$, $\tilde{\s}_i$, and $\tbv_i$ become $\0_{\tilde{p}}$, and changing the output $b$ by $q_i$ for $i=1,\ldots,N$, as the diagonal elements of $(A+D)^{-1}$.  

\begin{algorithm}
  \caption{Trace of $L^{-1}(\tilde{A}+\tilde{D})L^{-T}$}
  \label{alg:traceADtAtD}
  \begin{algorithmic}
  \STATE{\textbf{Input:} Representation $\c_i$, $\s_i$, $\w_i\in\R^p$ and $\f\in\R^N$ for $L$ in \cref{eq:chol_L}; $\tilde{\c}_i$, $\tilde{\s}_i$, and $\tbv_i\in\R^{\tilde{p}}$ for $\tilde{A}$ in \cref{eq:tildeA}; and $\tilde{\d}\in\R^N$ such that $\tilde{D}=\diag(\tilde{\d})$;}  
  \STATE{\textbf{Output:} $b\in\R$ such that $b=\tr(L^{-1}(\tilde{A}+\tilde{D})L^{-T})$;} 
  \STATE{Initialize $P\leftarrow \0_{p\times p}$, $R\leftarrow \0_{\tilde{p}\times p}$;}
  \FOR{$i=1,\ldots,N$}
    \STATE{$\p\leftarrow P\c_i$; \quad $\r\leftarrow R^T\tilde{\c}_i$;}
    \STATE{$q_i\leftarrow f_i^{-2}(\c_i^T\p-2\r^T\c_i+\tilde{\c}_i^T\tbv_i+\tilde{d}_i)$;}
    \STATE{$R\leftarrow \tilde{S}_i\mbk{R+f_i^{-1}(\tbv_i-R \c_i)\w_i^T}S_i$;}
    \STATE{$P\leftarrow S_i\cbk{
  P + f_i^{-1}\mbk{(\r-\p)\w_i^T + \w_i(\r^T-\p^T)} + q_i\w_i\w_i^T}S_i$;}
  \ENDFOR
  \STATE{$b\leftarrow\sum_{i=1}^{N}q_i;$}
  \end{algorithmic}
\end{algorithm}

\Cref{alg:diagADinv,alg:traceADtAtD} do not use the implicit representation of $L^{-1}$ in \cref{eq:chol_Linv}.
Therefore, when some $d_i$ are close to zero, these two algorithms are numerically stable. 
In contrast,  \cite[Algorithm 4.5]{andersen2020smoothing} may fail in this case, which uses the GR of $L^{-1}$ and has a computational complexity of $\calO(Np^3)$.
Although one can compute $\b=\sum_{i=1}^N\norm{L^{-1}\e_i}^2$ based on the GR of $L$ via \cite[Algorithm C.3]{andersen2020smoothing}, it costs $\calO(N^2 p)$ flops.

\section{Experimental results}
\label{sec:experiments}
In this section, we run Monte Carlo simulations to test the numerical stability, accuracy and efficiency of the proposed implementation of algorithms using the GvR obtained by the analytic form ($\GvR$) or the GvR obtained by numerical computation from the GR of a kernel matrix via \cref{eq:gen2giv_G,eq:gen2giv_vhat} ($\GvRt$). The proposed implementation of algorithms is compared with 
(i) the MATLAB built-in functions for the quantities in \Cref{tab:krm_algorithms} ($\MRef$); and 
(ii) GR-based algorithms in \cite{andersen2020smoothing}, where $\tr(M_{\boldeta,\gamma}^{-1})$ is computed via \cite[Algorithm 4.5]{andersen2020smoothing} ($\GR$) or by more stable but computationally more expensive $\sum_{i=1}^N\norm{L_{\boldeta,\gamma}^{-1}\e_i}^2$ via \cite[Algorithm C.3]{andersen2020smoothing} ($\GRs$). 
We choose GCV \cref{eq:GCV_opt} as hyper-parameter optimization criterion, done by first choosing an initial point of $(\boldeta,\gamma)$ via grid search, and then applying the MATLAB function $\mathtt{fmincon}$ with the interior-point algorithm to optimize the hyper-parameters.

We consider the KRM for the DT case. In particular, we consider the following two choices of the input $u(t)$ 
\begin{enumerate}
  \item[(S1)] unit impulse signal $u(t)=\1(t=0)$ in Example~\ref{ex:unitImp}, and
  \item[(S2)] exponential signal $u(t)=e^{-\alpha t}$ with $\alpha=0.5$ in \cref{eq:input-exp};
\end{enumerate}
and the DC kernel \cref{eq:DC}: In (S1), $\Psi_{\boldeta}=K_{\boldeta}^{\DC}\in\calS_{N,1}$ such that $\Psi_{\boldeta}$ has GR \cref{eq:gen_DC} and analytic form of GvR \cref{eq:Giv_DC_equspace}; and
in (S2), $\Psi_{\boldeta}\in\calS_{N,2}$ has GR \cref{eq:output_kernel_input_exp_GR} and GvR \cref{eq:output_kernel_input_exp_GvR_DT} in (S2).
The noise $\varepsilon(t)$ follows from a zero-mean Gaussian distribution with signal-to-noise ratio $\SNR=10$.

\begin{figure}[t]
  \centering
  \includegraphics[width=1.0\textwidth]{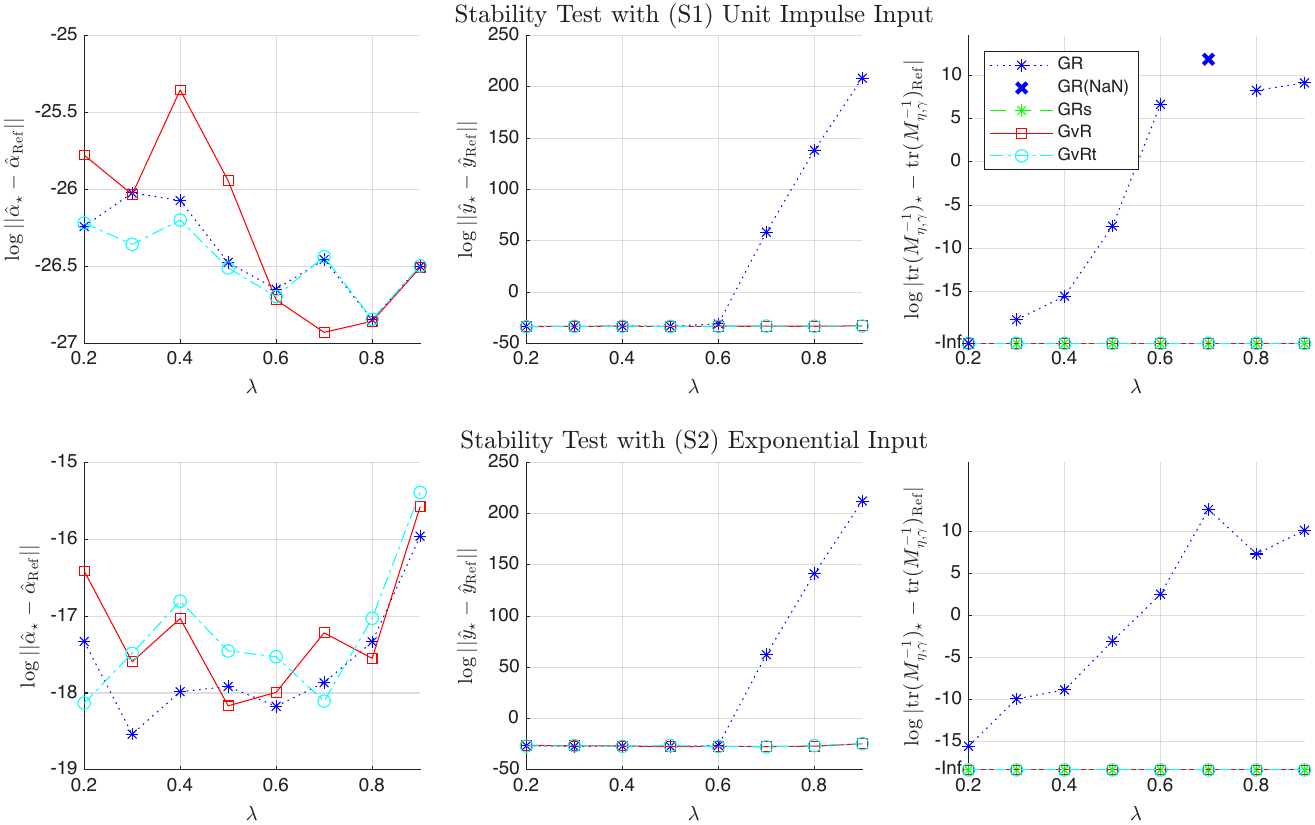}
  \caption{The logarithms of the averaged difference norms with respect to $\lambda$ using methods $\star\in\{\GR,\GRs,\GvR,\GvRt\}$ while fixing $(c,\rho,\gamma)=(1,0.6,10^{-4})$. 
  In the first two columns, $\GR$ and $\GRs$ are the same.  
  The first row uses the unit impulse input (S1) where $\GR$ returns NaN when $\lambda=0.7$, and the second row uses the exponential input $u(t)=e^{-0.5t}$ (S2).
  The experiments are repeated 80 times.}
  \label{fig:stab_alpha05}
\end{figure}

\begin{table}[t]
\centering
\caption{Averaged model fits for accuracy and efficiency test.}
\label{tab:sim_avg_fit}
\begin{tabular}{cc|ccccc}
 &  & $\GR$ & $\GRs$ & $\GvR$ & $\GvRt$ & $\MRef$ \\ \hline
Accuracy & Unit Impulse (S1) & 92.38 & 96.00 & 98.14 & 98.08 & 98.13 \\ 
& Exponential (S2) & 67.57 & 71.21 & 74.45 & 73.86 & 74.06 \\
Efficiency & Unit Impulse (S1) & 91.95 & 95.80 & 98.14 & 97.68 & 98.14 \\ 
& Exponential (S2) & 82.20 & 82.43 & 83.63 & 83.38 & 83.64 \\
\hline
\end{tabular}
\end{table}

\begin{figure}[t]
  \centering
  \includegraphics[width=1.0\textwidth]{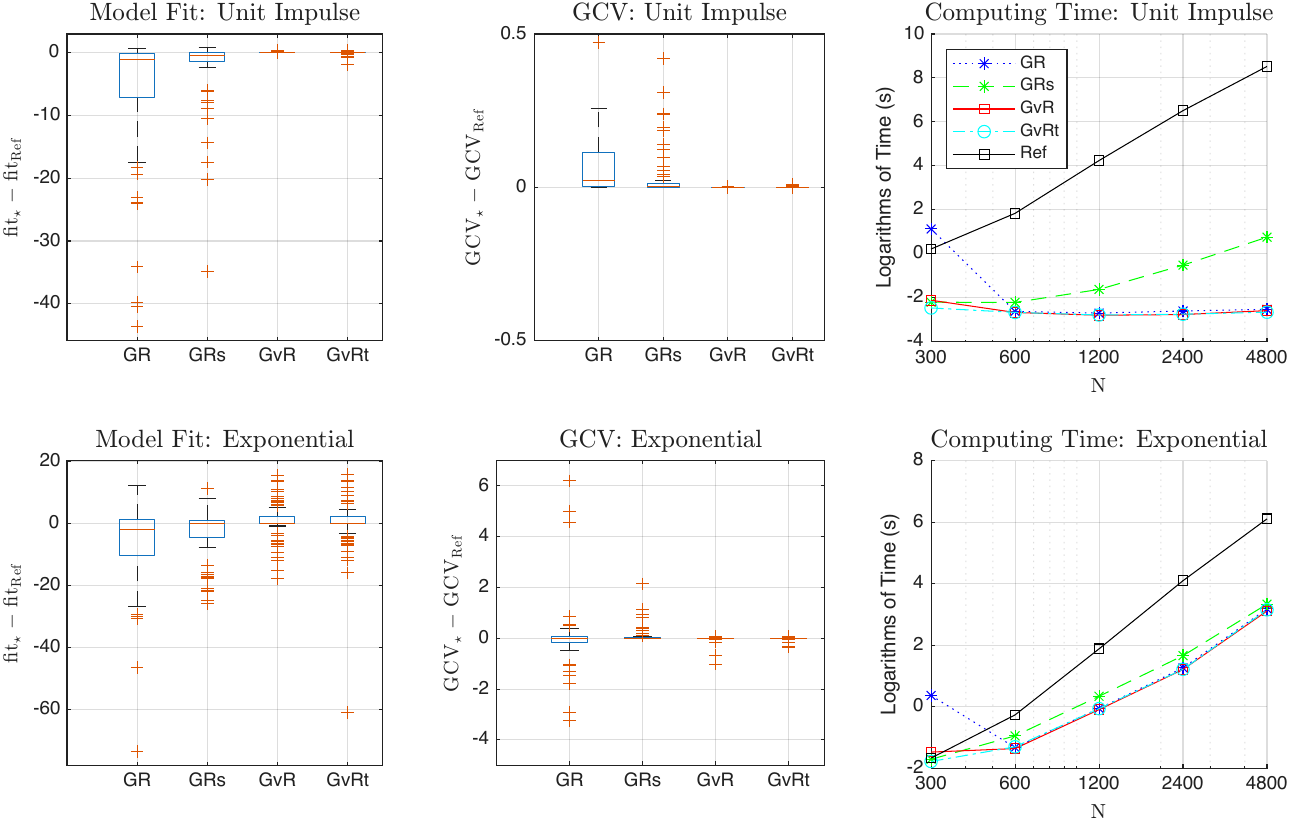}
  \caption{The first column shows the distributions of the model fit difference for $\GR$, $\GRs$, $\GvR$, and $\GvRt$, while the second column shows the distributions of the optimized GCV objectives for the four methods over 80 repeated experiments.
  The third column displays the logarithms of the averaged computation time (in seconds) for evaluating the GCV 200 times with respect to $N$ over 10 repeats, where the simulation is run on a Mac mini with Apple M4 Pro chip with 14-core CPU and 48 GB unified memory.
  }
  \label{fig:acc_eff_alpha05}
\end{figure}

\subsection{Stability and Accuracy Test}\label{sec:sim_stability_accuracy}
We generate 80 DT linear systems of 10th order with the moduli of all the poles within $[0.1,0.9]$.
Let $N=600$.
We test the numerical stability of computing $\hat{\boldalpha}$, $\hat{\y}=\Psi_{\boldeta}\hat{\boldalpha}$, and $\tr(M_{\boldeta,\gamma}^{-1})$ by fixing the hyper-parameter $(c,\rho,\gamma)=(1,0.6,10^{-4})$ and varying $\lambda=0.2,0.3,\ldots,0.9$.
\Cref{fig:stab_alpha05} displays the logarithms of the 80-repetition averaged $\norm{\hat{\boldalpha}_{\star}-\hat{\boldalpha}_{\MRef}}$ and $\norm{\hat{\y}_{\star}-\hat{\y}_{\MRef}}$ for $\star\in\{\GR,\GvR,\GvRt\}$ and $\abs{\tr(M_{\boldeta,\gamma}^{-1})_{\star}-\tr(M_{\boldeta,\gamma}^{-1})_{\MRef}}$ for $\star\in\{\GR,\GRs,\GvR,\GvRt\}$.
The performance of $\GR$ and $\GvR$ are similar in computing $\hat{\boldalpha}$ in both inputs (S1) and (S2), but GvR-based algorithms $\GvR$ and $\GvRt$ are more stable in computing $\hat{\y}$ and $\tr(M_{\boldeta,\gamma}^{-1})$ when $\lambda>\rho$, aligned with previous analysis, although $\GvRt$ is based on top of the GR.
Notably, when $\lambda=0.7$, $\GR$ fails to compute $\tr(M_{\boldeta,\gamma}^{-1})$ and return NaN due to the ill-conditioned GR of $L_{\boldeta,\gamma}^{-1}$. 
We also test the stability by varying $\alpha=0.5,1.0,1.5$ in (S2) and the results are displayed in \Cref{fig:stab_changeAlpha} of the Appendix~\ref{sec:sim_additional}.
For larger $\alpha$, i.e., faster decay rate, $\GR$-based algorithms become more unstable for a wider range of $\lambda$.

Next, we compare the accuracy by the model fit for the impulse response $\hat{g}$:
\begin{align*}
    \fit=100\left(1-\left[\frac{\sum_{k=1}^N|g_0(k)-\hat{g}(k)|}{\sum_{k=1}^N|g^0(k)-\bar{g}^0|}\right]^{1/2}\right),\quad \bar{g}^0=\frac{1}{N}\sum_{k=1}^Ng^0(k),
\end{align*}
where $g^0(k)$ and $\hat{g}(k)$ are the true and the estimated impulse response at the $k$th order, respectively.
Fix $c=1$, and use grid search to optimize $(\rho,\lambda,\gamma)$.
The averaged model fits out of 80 simulation runs for the five methods are shown in \Cref{tab:sim_avg_fit}, where our $\GvR$ and $\GvRt$ have values closer to $\MRef$ than $\GR$ and $\GRs$.
To display the deviation, the first two columns of \Cref{fig:acc_eff_alpha05} show the distributions of the model fit difference $(\fit_{\star}-\fit_{\MRef})$ and optimized GCV objectives $(\GCV_{\star}-\GCV_{\MRef})$ for $\star\in\{\GR,\GRs,\GvR,\GvRt\}$.
For model (S1), our $\GvR$ and $\GvRt$ have the highest accuracy with minimal deviation from the baseline result $\MRef$, whereas $\GR$ and $\GRs$ can significantly deviate from $\MRef$, suggesting their instability and poor accuracy.
For model (S2), although all three methods exhibit some deviation from $\MRef$, the values of $\GvR$ and $\GvRt$ still concentrate around zero more tightly than those of $\GR$ and $\GRs$ with shorter box height.
Comparing the GCV values, $\GvR$ and $\GvRt$ yield results closer to $\MRef$ with less variability than the $\GR$ and $\GRs$ in (S1), while in (S2), our GvR-based methods are still generally more accurate than $\GR$.
It is worth mentioning that in (S2), the maximum value $(\GCV_{\GR}-\GCV_{\MRef})$ over the 80 simulation runs is 8943, a huge derivation unplotted in \Cref{fig:acc_eff_alpha05}.
The regime of hyper-parameters in which the $\GR$ and $\GRs$-based algorithms face numerical instability has a neglectable effect on the performance of $\GvRt$, showing the robustness of the procedure \cref{eq:gen2giv_G,eq:gen2giv_vhat}.

\subsection{Efficiency Test}\label{sec:sim_efficiency}
To test the efficiency, we generate data from the first 10 systems out of the 80 systems in \Cref{sec:sim_stability_accuracy} with $N=300,600,1200,2400,4800$.
For each $N$, we identify the system and record the average accumulative computation time for evaluating the GCV \cref{eq:GCV_opt} 200 times in the initial grid search (with size 200).
The averaged model fits are reported in \Cref{tab:sim_avg_fit}.
The third column of \Cref{fig:acc_eff_alpha05} displays the averaged computation time with respect to varying $N$, indicating the superior efficiency of our $\GvR$ and $\GvRt$ compared to $\MRef$.

\section{Conclusions}
\label{sec:conclusions}

The existing works on efficient algorithms for the KRSysId almost all exploit the semiseparable structure of the kernel matrix and are based on its GR. 
However, the GR-based algorithms can be numerically unstable or lose the accuracy significantly. 
To overcome this issue, we proposed to use the GvR of semiseparable kernels in this paper. 
In particular, we first derived the GvR of some commonly used kernel matrices and output kernel matrices in the KRSysID.
We then derived the GvR-based algorithms and showed that they have a computational complexity of at most $\calO(Np^2)$. 
Monte Carlo simulation result shows that the proposed implementation of algorithms is more stable, more accurate, and more efficient than the state-of-art GR-based ones.

\appendix

\section{Details of the Two Examples in \Cref{sec:intro}}\label{sec:two_eg}
This section provides details of the two examples mentioned in \Cref{sec:intro} about the matrix-vector product $K_{\boldeta}\x$ for $\x\in\R^N$ and inverse Cholesky factor $L_{\boldeta,\gamma}^{-1}$ of $M_{\boldeta,\gamma}=K_{\boldeta}+\gamma I_N$ for $p$-GRS $K_{\boldeta}\in\R^{N\times N}$
\begin{align*}
    K_{\boldeta} = \tril(UV^T) + \triu(VU^{T},1)
\end{align*}where its GR $U=[\boldmu_1,\ldots,\boldmu_N]^T$ and $V=[\boldnu_1,\ldots,\boldeta_N]^T\in\R^{N\times p}$.
We use $(\cdot)_{\GR}$, $(\cdot)_{\GvR}$, and $(\cdot)_{\MRef}$ to denote the results using GR-based algorithms, GvR-based algorithms and MATLAB built-in operations with 16 decimal digits of precision (double precision), and $(\cdot)_{\GR(\High)}$ for GR-based algorithms with 50 decimal digits.

In the first example, we compute $\y=K_{\boldeta}\x$.
To be specific, recall $x_i$ and $y_i$ the $i$th element of $\x$ and $\y$, respectively.
Applying \cite[Algorithm 4.1]{andersen2020smoothing} to compute $\y=K_{\boldeta}\x$ yields that
\begin{equation*}\label{eq:two_eg_Kx}
    \begin{aligned}
        y_i &= \boldmu_i^T \bar{\boldnu}_i + \boldnu_i^T \bar{\boldmu}_i, \\
    \bar{\boldmu}_i & = \sum_{j=i+1}^N \boldmu_j x_j = \bar{\boldmu}_{i-1} - \boldmu_i x_i, \quad \bar{\boldmu}_0=U^T\x, \ \bar{\boldmu}_{N}=\0_p,\\
    \bar{\boldnu}_i &= \sum_{j=i}^{i} \boldnu_j x_j = \bar{\boldnu}_{i-1} + \boldnu_i x_i, \quad \bar{\boldnu}_0=\0_p,
    \end{aligned}
\end{equation*}for $i=1,\ldots,N$.
Recall that we use $K_{\boldeta}^{\DC}$ with $N=5$, $t_i=i$, $\lambda=0.1$, $\rho=10^{-7}$, and $c=1$ in \cref{eq:gen_DC}.
In this case, since $\lambda\rho=10^{-8}$, $\lambda/\rho=10^6$, we have 
\begin{align*}
    U &= \begin{bmatrix}
        10^{-8} & 10^{-16} & 10^{-24} & 10^{-32} & 10^{-40}
    \end{bmatrix}^T, \\
    V &= \begin{bmatrix}
        10^{6} & 10^{12}&10^{18}& 10^{24}&10^{30}
    \end{bmatrix}^T.
\end{align*}
By the above procedure, the elements $y_i$ of $\y=K_{\boldeta}^{\DC}\x$ are
\begin{equation}\label{eq:two_eg_y15}
    \begin{aligned}
    y_1 &
    = 10^{-8}(10^6x_1) +10^6(10^{-16}x_2 + 10^{-24}x_3 + 10^{-32}x_4 + 10^{-40}x_5), \\
    y_2 &
    = 10^{-16}(10^6x_1+10^{12}x_2) +10^{12}(10^{-24}x_3 + 10^{-32}x_4 + 10^{-40}x_5), \\
    y_3 &
    = 10^{-24}(10^6x_1+10^{12}x_2+10^{18}x_3) +10^{18}(10^{-32}x_4 + 10^{-40}x_5), \\
    y_4 &
    = 10^{-32}(10^6x_1+10^{12}x_2+10^{18}x_3+10^{24}x_4) +10^{24}(10^{-40}x_5), \\
    y_5 &=10^{-40}(10^6 x_1 + 10^{12}x_2 + 10^{18}x_3 + 10^{24} x_4 + 10^{30} x_5).
    \end{aligned}
\end{equation}
Their products span enormous range.
For example, when $\x=[-1,1,-1,1,-1]^T$ with mild magnitude, the GR-based result has a much larger relative error measured by vector 2-norm compared to our GvR-based one using \Cref{alg:Ax}
\begin{align*}\label{eq:two_eg_Kx_relative_error}
    \frac{\norm{\y_{\GR}-\y_{\MRef}}_2}{\norm{\y_{\MRef}}_2} \approx \num{6.224530e+07}, \quad \frac{\norm{\y_{\GvR}-\y_{\MRef}}_2}{\norm{\y_{\MRef}}_2} \approx \num{1.421267e-08}.
\end{align*}

In the second example, for a $p$-GRS matrix $K_{\boldeta}\in\R^{N\times N}$, then by \cite[Algorithm 4.3]{andersen2020smoothing}, the Cholesky factor $L_{\boldeta,\gamma}$ of $M_{\boldeta,\gamma}=K_{\boldeta}+ \gamma I_N=L_{\boldeta,\gamma}L_{\boldeta,\gamma}^T$ has GR 
\begin{align*}\label{eq:two_eg_L_GR}
    L_{\boldeta,\gamma} = \tril(UW^T,-1)+\diag(\c), \quad W\in\R^{N\times p}, \ \c\in\R_{++}^p,
\end{align*}and by \cite[Algorithm 4.4]{andersen2020smoothing}, its inverse $L_{\boldeta,\gamma}^{-1}$ has GR
\begin{align*}\label{eq:two_eg_Linv_GR}
    L_{\boldeta,\gamma}^{-1} = \tril(YZ^T,-1)+\diag(\c)^{-1},
\end{align*}where $Y=L_{\boldeta,\gamma}^{-1}U$ and $Z=L_{\boldeta,\gamma}^{-T}W(Y^TW-I_p)^{-1}$.
However, when $\gamma\approx 0$, the condition numbers $\kappa(M_{\boldeta,\gamma})$ and $\kappa(Y^T W - I_p)$ will be extremely large.
As a result, the implicit GR $(Y,Z)$ is inaccurate.
In our example, $K_{\boldeta}^{\SS}$ with $N=5$, $t_i=i$, $p=2$, $\rho=0.1$, and $c=1$ in \cref{eq:gen_SS}, and $\gamma=10^{-8}$.
Then
\begin{align*}\label{eq:two_eg_M_Y_cond}
    \kappa(M_{\boldeta,\gamma}) = \num{3.191245e+04}, \quad \kappa(Y^TW-I_2) \approx \num{6.890193e+16},
\end{align*}
and as a consequence, the relative error measured by spectral norm
\begin{align*}\label{eq:two_eg_Z_cond}
    \frac{\norm{Z_{\GR}-Z_{\GR(\High)}}_2}{\norm{Z_{\GR(\High)}}_2} \approx \num{1.002317}.
\end{align*}

Moreover, having accurate $Y$ and $Z$ 
does not allow us to accurately reconstruct, for $1\leq j < i \leq N$, the $(i,j)$-entry $\y_i^T\z_j$ of $\tril(L_{\boldeta,\gamma}^{-1},-1)$, where $\y_i,\z_j\in\R^p$ are the $i$th and $j$th column of $Y^T$ and $Z^T$, respectively.
To see this, recall that the relative condition number associated with this inner product \cite[Section~3]{higham2002accuracy} is
\begin{align}\label{eq:inner_cond}
    \frac{|\y_i|^T|\z_j|}{|\y_i^T\z_j|},
\end{align}where $\abs{\cdot}$ takes the element-wise absolute value, which may be extremely large as well,
making the computation of entries of $\tril(L_{\boldeta,\gamma}^{-1},-1)$ unreliably even if $Y$ and $Z$ are accurate to double precision.
In our $K_{\boldeta}^{\SS}$ example, the matrix of \cref{eq:inner_cond} when computing $\tril(L_{\boldeta,\gamma}^{-1},-1)$ using high precision $Y_{\GR(\High)}$ and $Z_{\GR(\High)}$ is
\begin{align}\label{eq:two_eg_cond_inner}
    \begin{bmatrix}
        \num{0.00} & \num{0.00} & \num{0.00} & \num{0.00} & \num{0.00} \\
        \num{2.50e+06} & \num{0.00} & \num{0.00} & \num{0.00} & \num{0.00} \\
        \num{1.04e+06} & \num{3.26e+10} & \num{0.00} & \num{0.00} & \num{0.00} \\
        \num{1.04e+06} & \num{1.81e+10} & \num{8.94e+13} & \num{0.00} & \num{0.00} \\
        \num{1.04e+06} & \num{1.81e+10} & \num{5.06e+13} & \num{3.20e+16} & \num{0.00}
    \end{bmatrix}.
\end{align}
If we compute $Y_{\GR(\High)}$ and $Z_{\GR(\High)}$ and round them to double precision to form $\tril(L_{\boldeta,\gamma}^{-1},-1)_{\GR}$, then the relative error is much larger than computing $\bar{\c}_i$, $\bar{S}_i$, and $\bar{\w}_i$ in \cref{eq:inv_rep} with only double precision then forming $\tril(L_{\boldeta,\gamma}^{-1},-1)_{\GvR}$ explicitly in \cref{eq:chol_Linv}. In particular, we have
\begin{equation*}\label{eq:two_eg_trilLinv_GR}
    \begin{aligned}
        \frac{\norm{\tril(L_{\boldeta,\gamma}^{-1},-1)_{\GR}-\tril(L_{\boldeta,\gamma}^{-1},-1)_{\MRef}}_2}{\norm{\tril(L_{\boldeta,\gamma}^{-1},-1)_{\MRef}}_2}&\approx \num{1.945209},  \\
        \frac{\norm{\tril(L_{\boldeta,\gamma}^{-1},-1)_{\GvR}-\tril(L_{\boldeta,\gamma}^{-1},-1)_{\MRef}}_2}{\norm{\tril(L_{\boldeta,\gamma}^{-1},-1)_{\MRef}}_2}&\approx \num{1.050701e-11}.
    \end{aligned}
\end{equation*}

\section{Proofs}\label{sec:proofs}

\subsection{Proof of \Cref{prop:Giv_SS_DC_TC}}\label{sec-proof:Giv_SS_DC_TC}
\begin{proof}[Proof for SS kernels \cref{eq:Giv_SS}]\label{proof:Giv_SS}
  Recall the GR of SS kernels \cref{eq:gen_SS}
    \begin{align*}
        \mu_{i,1}=-\frac{\rho^{3 t_i}}{6}, \ \nu_{i,1}=1, \
        \mu_{i,2}=\frac{\rho^{2 t_i}}{2},\ \nu_{i,2}= \rho^{t_i},\quad i=1,\ldots,N.
    \end{align*}
  We start with $k=1$, i.e., find $c_{i.1}$, $s_{i,1}$, and $\hv_{i,1}$,
  At step $i=N$, since $\nu_{N,1}=1> 0$, the signs of $\hv_{N,1}$ and $\mu_{N,1}$ should be the same, so by \cref{eq:gen2giv_vhat}, $\hv_{N,1}=\nu_{N,1}\mu_{N,1}=-\rho^{3 t_N}/6$.
  At step $i=N-1$, by \cref{eq:gen2giv_G}, rotation components $c_{N-1,1}$ and $s_{N-1,1}$ satisfy
  \begin{align*}
    \begin{bmatrix}
      c_{N-1,1} & -s_{N-1,1} \\ s_{N-1,1} & c_{N-1,1}
    \end{bmatrix}\begin{bmatrix}
      r_{N-1,1} \\ 0
    \end{bmatrix} = 
    \begin{bmatrix}
      \mu_{N-1,1} \\ \mu_{N,1}
    \end{bmatrix}
    =\frac 1 6\begin{bmatrix}
      -\rho^{3 t_{N-1}} \\ -\rho^{3 t_{N}}
    \end{bmatrix},
  \end{align*} 
  where 
  \begin{align*}
    r_{N-1,1} &= \sqrt{\mu_{N,1}^2+\mu_{N-1,1}^2}=\frac{1}{6}\sqrt{\rho^{6 t_N}+\rho^{6 t_{N-1}}}, \\
    c_{N-1,1} &= \frac{\mu_{N-1,1}}{r_{N-1,1}} = \frac{-\rho^{3 t_{N-1}}}{\sqrt{\rho^{6 t_N}+\rho^{6 t_{N-1}}}}, \\
    s_{N-1,1} &= \frac{\mu_{N,1}}{r_{N-1,1}} = \frac{-\rho^{3 t_N}}{\sqrt{\rho^{6 t_N}+\rho^{6 t_{N-1}}}}.
  \end{align*}
  By \cref{eq:gen2giv_vhat}, as $c_{N-1,1}$, $\mu_{N-1,1}<0$ and $\nu_{N-1,1}>0$,
  \begin{align*}
      \hv_{N-1,1}=\nu_{N-1,1} r_{N-1,1}=\frac{1}{6}\sqrt{\rho^{6 t_N}+\rho^{6 t_{N-1}}}.
  \end{align*}
  At step $i=N-2,\ldots,1$, $r_{i,1}=\sqrt{\sum_{j=i}^N \mu_{j,1}^2}=\frac{1}{6}\sqrt{\sum_{j=i}^N \rho^{6 t_j}}$. 
  Thus 
  \begin{align*}
    c_{i,1}&=\frac{\mu_{i,1}}{r_{i,1}}=\frac{-\rho^{3 t_i}}{\sqrt{\sum_{j=i}^N \rho^{6 t_j}}}, \ s_{i,1}=\frac{r_{i+1,1}}{r_{i,1}}=\frac{\sqrt{\sum_{j=i+1}^N \rho^{6 t_j}}}{\sqrt{\sum_{j=i}^N \rho^{6 t_j}}}, \\
    \hv_{i,1}&=\nu_{i,1}r_{i,1}=\frac{1}{6}\sqrt{\sum_{j=i}^N \rho^{6 t_j}},
  \end{align*}
  by $c_{i,1},\mu_{i,1}<0$ and $\nu_{i,1}>0$.
  Note that the above equations are compatible with the case $c_{N-1,1}$ and $\hv_{N-1,1}$, but not $s_{N-1,1}$ and $\hv_{N,1}$ due to the sign. 
  So we write
    \begin{align*}
      s_{i,1}&=\frac{(-1)^{\1(i=N-1)} \sqrt{\sum_{j=i+1}^N \rho^{6 t_j}}}{\sqrt{\sum_{j=i}^N \rho^{6 t_j}}}, \ i=1,\ldots,N-1, \\
      \hv_{i,1}& =\frac{(-1)^{\1(i=N)}}{6} \sqrt{\sum_{j=i}^N \rho^{6 t_j}},\ i=1,\ldots,N.
  \end{align*}
  For $k=2$, we derive $c_{i,2},s_{i,2}$, and $\hv_{i,2}$.
  At step $i=N$, since $\nu_{N,2},\mu_{N,2}>0$, by \cref{eq:gen2giv_vhat}, $\hv_{N,2}=\nu_{N,2}{\mu_{N,2}}=\rho^{t_N} \cdot \rho^{t_N}/2=\rho^{3 t_N}/2$.
  At step $i=N-1,\ldots,1$, we have 
  \begin{align*}
    \begin{bmatrix}
      c_{i,2} & -s_{i,2} \\ s_{i,2} & c_{i,2}
    \end{bmatrix}\begin{bmatrix}
      r_{i,2} \\ 0
    \end{bmatrix} = 
    \begin{bmatrix}
      \mu_{i,2} \\ \mu_{i+1,2}
    \end{bmatrix}
    =\begin{bmatrix}
      \rho^{2 t_{i}}/2 \\ \rho^{2 t_{i+1}}/2
    \end{bmatrix},
  \end{align*}
  where $r_{i,2} = \sqrt{\sum_{j=i}^N \mu_{j,2}^2}=(1/2)\sqrt{\sum_{j=i}^N \rho^{4 t_j}}$, and
  \begin{align*}
    c_{i,2} &= \frac{\mu_{i,2}}{r_{i,2}}=\frac{\rho^{2 t_i}}{\sqrt{\sum_{j=i}^N \rho^{4 t_j}}}, \quad
    s_{i,2} = \frac{r_{i+1,2}}{r_{i,2}}=\frac{\sqrt{\sum_{j=i+1}^N \rho^{4 t_j}}}{\sqrt{\sum_{j=i}^N \rho^{4 t_j}}}.
  \end{align*}
  Since $c_{i,2}$, $\nu_{i,2}$, and $\mu_{i,2}>0$, by \cref{eq:gen2giv_vhat}, we have 
  \begin{align*}
      \hv_{i,2}&= \nu_{i,2} r_{i,2}=\frac{\rho^{ t_i}}{2}\sqrt{\sum_{j=i}^N \rho^{4 t_j}},
  \end{align*}which is compatible with the case $i=N$.
  Thus \cref{eq:Giv_SS} is proved for both $k=1,2$.
\end{proof}

\begin{proof}[Proof for DC kernels \cref{eq:Giv_DC}]\label{proof:Giv_DC}
Recall the GR \cref{eq:gen_DC} of DC kernels $\mu_i=(\lambda\rho)^{t_i}>0$ and $\nu_i=(\lambda/\rho)^{t_i}>0$.
    At step $i=N$, as $\nu_N,\mu_N>0$, by \cref{eq:gen2giv_vhat}, $\hv_N=\nu_N \mu_N=({\lambda}/{\rho})^{t_N}(\lambda\rho)^{t_N}=\lambda^{2t_N}$.
    At step $i=N-1,\ldots,1$,
    $r_i = \sqrt{\sum_{j=i}^N \mu_{j}^2} = \sqrt{\sum_{j=i}^N (\lambda\rho)^{2t_{j}}}$.
    By the Givens rotation \cref{eq:gen2giv_G},
    \begin{align*}
      \begin{bmatrix}
        c_{i} & -s_{i} \\ s_{i} & c_{i}
      \end{bmatrix}\begin{bmatrix}
        r_{i} \\ 0
      \end{bmatrix} = 
      \begin{bmatrix}
        \mu_{i} \\ r_{i+1}
      \end{bmatrix}
      =\begin{bmatrix}
        (\lambda\rho)^{t_{i}} \\ 
        \cbk{\sum_{j=i+1}^N (\lambda\rho)^{2t_{j}}}^{1/2}
      \end{bmatrix},
    \end{align*} 
    we can compute $c_i$ and $s_i$ by
    \begin{align*}
        c_i&=\frac{\mu_i}{r_i}=\frac{(\lambda\rho)^{t_i}}{\sqrt{\sum_{j=i}^N (\lambda\rho)^{2t_j}}}, \quad
        s_i=\frac{r_{i+1}}{r_i}=\frac{\sqrt{\sum_{j=i+1}^N (\lambda\rho)^{2t_j}}}{\sqrt{\sum_{j=i}^N (\lambda\rho)^{2t_j}}}.
    \end{align*}
    Since $c_i$, $\mu_i$, $\nu_i > 0$, we have $\hv_i>0$, thus by \cref{eq:gen2giv_vhat},
    \begin{align*}
        \hv_i = \nu_i r_i = \sbk{\frac{\lambda}{\rho}}^{t_i}\sqrt{\sum_{j=i}^N (\lambda\rho)^{2t_j}}, \quad i=N-1,\ldots,1,
    \end{align*}which is compatible with the case $i=N$. 
    This completes the proof of \cref{eq:Giv_DC}.
  
\end{proof}

\begin{remark}
In practice, suppose $t_i=T i$ for $i=1,\ldots,N$ and sampling time $T>0$, then the equi-spaced version of GvR \cref{eq:Giv_DC} for $K_{\boldeta}^{\DC}$ is
\begin{equation}\label{eq:Giv_DC_equspace}
\begin{aligned}
        c_i &= \sqrt{\frac{1-(\lambda\rho)^{2T}}{1-(\lambda\rho)^{2T(N-i+1)}}}, \quad
        s_i = (\lambda\rho)^{T} \sqrt{\frac{1-(\lambda\rho)^{2T(N-i)}}{1-(\lambda\rho)^{2T(N-i+1)}}}, \\
        \hv_{\ell} &= \lambda^{2T\ell}\sqrt{\frac{1-(\lambda\rho)^{2T(N-\ell+1)}}{1-(\lambda\rho)^{2T}}},
\end{aligned}
\end{equation}
for $i=1,\ldots,N-1$ and $\ell=1,\ldots,N$.
The equi-spaced version of GvR \cref{eq:Giv_SS} for $K_{\boldeta}^{\SS}$ is
\begin{equation}\label{eq:Giv_SS_equspace}
        \begin{aligned}
         & \c_{i}^T=\begin{bmatrix}
            -\sqrt{\frac{1-\rho^{6T}}{1-\rho^{6T(N-i+1)}}} &
            \sqrt{\frac{1-\rho^{4T}}{1-\rho^{4T(N-i+1)}}}
        \end{bmatrix},\\
        & \s_{i}^T=\begin{bmatrix}
            (-1)^{\1(i=N-1)} \rho^{3T} \sqrt{\frac{1-\rho^{6T(N-i)}}{1-\rho^{6T(N-i+1)}}} & 
            \rho^{2T} \sqrt{\frac{1-\rho^{4T(N-i)}}{1-\rho^{4T(N-i+1)}}}
        \end{bmatrix}, \\
        & \hbv_{\ell}^T=\begin{bmatrix}
            \frac{(-1)^{\1(i=N)} \rho^{3T\ell}}{6}\sqrt{\frac{1-\rho^{6T(N-\ell+1)}}{1-\rho^{6T}}} &
            \frac{\rho^{3T\ell}}{2} \sqrt{\frac{1-\rho^{4T(N-\ell+1)}}{1-\rho^{4T}}}
        \end{bmatrix},
      \end{aligned}
\end{equation}
for $i=1,\ldots,N-1$ and $\ell=1,\ldots,N$.    
\end{remark}

\subsection{Proof of \Cref{prop:Giv_output}}\label{sec-proof:Giv_output}
    
Recall the exponential input signal $u(t)=e^{-\alpha t}$ and $u(t-b)=e^{-\alpha t}e^{\alpha b}$ for $\alpha\in\R$ in \cref{eq:input-exp} and DC kernel \cref{eq:DC}. 
    In this case, \cref{eq:input_condition_GR_output_kernel} holds with $\pi_1(t)=e^{-\alpha t}$ and $\rho_1(b)=e^{\alpha b}$.
    Since $K_{\boldeta}\in\calG_{N,1}$, by \cite[Theorem 3]{CA21}, the output kernel matrix $\Psi_{\boldeta}\in\calG_{N,2}\subset\calS_{N,2}$.
    We first  derive the GR of $\Psi_{\boldeta}$ in \Cref{sec-proof:Giv_output_GR} and then use \cref{eq:gen2giv_G,eq:gen2giv_vhat} to obtain its GvR in \Cref{sec-proof:Giv_output_GvR}.

\subsubsection{GR of $\Psi_{\boldeta}$}\label{sec-proof:Giv_output_GR}
We assume $T_{\lambda,\rho,\alpha}=\log(\lambda\rho)+\alpha\neq 0$ and $D_{\lambda,\rho,\alpha}=\log(\lambda/\rho)+\alpha\neq 0$.
Then $\Psi_{\boldeta}$ has GR
  \begin{align*}
    \boldmu_i = 
    \begin{bmatrix}
      \bar{\mu}_1(t_i) & \bar{\mu}_2(t_i)
    \end{bmatrix}^T, \quad 
    \boldnu_j = 
    \begin{bmatrix}
      \bar{\nu}_1(t_j) & \bar{\nu}_2(t_j)
    \end{bmatrix}^T,
  \end{align*}
  where $\bar{\mu}_2(t) = e^{-\alpha t}$, and
  \begin{equation}\label{eq:output_kernel_input_exp_GR}
    \begin{aligned}
      \bar{\mu}_1(t) &= \begin{cases}
      \frac{(\lambda\rho)^t - e^{-\alpha t}}{T_{\lambda,\rho,\alpha}} & \ \text{ (CT)}, \\
      \frac{e^{-\alpha t}-(\lambda\rho)^te^{T_{\lambda,\rho,\alpha}}}{T_{\lambda,\rho,\alpha}'} & \ \text{ (DT)},
    \end{cases} \quad
    \bar{\nu}_1(s) =\begin{cases}
      \frac{(\lambda/\rho)^s - e^{-\alpha s}}{D_{\lambda,\rho,\alpha}}& \ \text{ (CT)}, \\
      \frac{e^{-\alpha s}-(\lambda/\rho)^s e^{D_{\lambda,\rho,\alpha}}}{D_{\lambda,\rho,\alpha}'}& \ \text{ (DT)},
    \end{cases} \\
    \bar{\nu}_2(s) &= \begin{cases}
      \frac{(\lambda/\rho)^s -(\lambda\rho)^s+C_{\lambda,\rho,\alpha}(\lambda^{2s} e^{\alpha s} - e^{-\alpha s})}{D_{\lambda,\rho,\alpha}T_{\lambda,\rho,\alpha}} & \ \text{ (CT)}, \\
      \frac{e^{D_{\lambda,\rho,\alpha}}(\lambda/\rho)^s - e^{T_{\lambda,\rho,\alpha}}(\lambda\rho)^s + C_{\lambda,\rho,\alpha}'(e^{D_{\lambda,\rho,\alpha}+T_{\lambda,\rho,\alpha}}\lambda^{2s}e^{\alpha s}-e^{-\alpha s})}
      {D_{\lambda,\rho,\alpha}'T_{\lambda,\rho,\alpha}'} & \ \text{ (DT)},
    \end{cases}
    \end{aligned}
  \end{equation}
  with 
  $C_{\lambda,\rho,\alpha}={\log\rho}/(\log \lambda + \alpha)$, 
  $T_{\lambda,\rho,\alpha}'=1-e^{T_{\lambda,\rho,\alpha}}$, 
  $D_{\lambda,\rho,\alpha}'=1-e^{D_{\lambda,\rho,\alpha}}$,
  and 
  $C_{\lambda,\rho,\alpha}'=(e^{D_{\lambda,\rho,\alpha}}-e^{T_{\lambda,\rho,\alpha}})/(1-e^{D_{\lambda,\rho,\alpha}+T_{\lambda,\rho,\alpha}})$.

\begin{proof}
We drop the $\boldeta$ inside $\calK(t,s;\boldeta)$ and $\Psi(t,s;\boldeta)$ for simplicity.
  The DC kernel 
  \begin{align*}
    \calK^{\DC}(t,s)= \begin{cases}
      (\lambda\rho)^t (\lambda/\rho)^{s}, \text{ if } t\geq s, \\
      (\lambda/\rho)^{t} (\lambda\rho)^s, \text{ if } t<s,
    \end{cases}
  \end{align*}
  with $p'=1$, $\mu_1(t)=(\lambda\rho)^t$, and $\nu_1(s)=(\lambda/\rho)^s$.
  Then by \cite[Theorem 3]{CA21}, its output kernel $\Psi_{\boldeta}\in\calG_{N,2}$ with 
  \begin{align*}
    & \Psi(t,s)=\begin{cases}
      \bar{\mu}_1(t)\bar{\nu}_1(s)+\bar{\mu}_2(t)\bar{\nu}_2(s), & \text{ if } t\geq s, \\
      \bar{\nu}_1(t)\bar{\mu}_1(s)+\bar{\nu}_2(t)\bar{\mu}_2(s), & \text{ if } t < s,
    \end{cases} \\
    & \bar{\mu}_1(t) = \pi_1(t) f_{11}^{(1)}(t), \quad \bar{\mu}_2(t) = \pi_1(t), \\
    & \bar{\nu}_1(s) = \pi_1(s) f_{11}^{(2)}(s), \quad \bar{\nu}_2(s) = \bar{\ell}_1(s) + \bar{\rho}_1(s).
  \end{align*}
  We compute $f_{11}^{(1)}(t)$, $f_{11}^{(2)}(s)$, $\bar{\ell}_1(s)$, and $\bar{\rho}_1(s)$ by \cite[Equations~(22c)--(22f)]{CA21}.
  For the CT case,
  \begin{align*}
    f_{11}^{(1)} (t) &= \int_{0}^t \mu_1(b) \rho_1(b) \rmd b
    = \int_{0}^{t} (\lambda\rho)^b e^{\alpha b} \rmd b 
    = \frac{(\lambda\rho)^t e^{\alpha t} - 1}{T_{\lambda,\rho,\alpha}}, \\
    f_{11}^{(2)} (s) &= \int_{0}^s \nu_1(a) \rho_1(a) \rmd a = \int_{0}^{s} \sbk{\frac{\lambda}{\rho}}^a e^{\alpha a} \rmd a = \frac{(\lambda/\rho)^s e^{\alpha s} - 1}{D_{\lambda,\rho,\alpha}}, \\
    \bar{\ell}_1(s) &= -\bar{\nu}_1(s) f_{11}^{(1)}(s) 
    = -e^{-\alpha s} \mbk{ \frac{(\lambda/\rho)^s e^{\alpha s} - 1}{D_{\lambda,\rho,\alpha}} } \mbk{ \frac{(\lambda\rho)^s e^{\alpha s} - 1}{T_{\lambda,\rho,\alpha}}},
  \end{align*}where we define $T_{\lambda,\rho,\alpha}=\log(\lambda\rho)+\alpha\neq 0$ and $D_{\lambda,\rho,\alpha}=\log(\lambda/\rho)+\alpha \neq 0$, and by $1/D_{\lambda,\rho,\alpha}-1/T_{\lambda,\rho,\alpha}=2\log \rho/(D_{\lambda,\rho,\alpha}T_{\lambda,\rho,\alpha})$,
  \begin{align*}
    \bar{\rho}_1(s) &= \int_{0}^{s} \sbk{ \int_{0}^{s} \calK^{\DC}(b,a)u(s-a)\rmd a} \rho_1(b) \rmd b \\
    &=e^{-\alpha s} \int_{0}^{s} \sbk{ \int_{0}^{s} \calK^{\DC}(b,a)e^{\alpha a}\rmd a} e^{\alpha b} \rmd b \\
    &=e^{-\alpha s} \int_{0}^{s} \mbk{
      (\lambda\rho)^b \int_{0}^{b} \sbk{\frac{\lambda}{\rho}}^a e^{\alpha a}\rmd a + \sbk{\frac{\lambda}{\rho}}^b \int_{b}^{s} (\lambda\rho)^a e^{\alpha a}\rmd a
    }e^{\alpha b} \rmd b \\
    &= e^{-\alpha s} \int_{0}^{s} \mbk{\frac{\lambda^{2b} e^{\alpha b} - (\lambda\rho)^b}{D_{\lambda,\rho,\alpha}} + \frac{\sbk{{\lambda}/{\rho}}^b (\lambda\rho)^s e^{\alpha s} - \lambda^{2b} e^{\alpha b}}{T_{\lambda,\rho,\alpha}}} e^{\alpha b}\rmd b \\
    &=\frac{e^{-\alpha s}}{D_{\lambda,\rho,\alpha}}\mbk{\frac{\lambda^{2s} e^{2\alpha s} - 1}{2\log\lambda + 2\alpha} - \frac{\sbk{\lambda\rho}^s e^{\alpha s} - 1}{T_{\lambda,\rho,\alpha}}} \\
    &\qquad + \frac{\sbk{\lambda\rho}^s }{T_{\lambda,\rho,\alpha}} \cdot \frac{\sbk{\lambda/\rho}^s e^{\alpha s} - 1}{D_{\lambda,\rho,\alpha}} - \frac{e^{-\alpha s}}{T_{\lambda,\rho,\alpha}} \cdot \frac{\lambda^{2s} e^{2\alpha s} - 1}{2\log{\lambda} + 2\alpha} \\
    &= \frac{C_{\lambda,\rho,\alpha}(\lambda^{2s} e^{\alpha s} - e^{-\alpha s})+\lambda^{2s}e^{\alpha s} - 2(\lambda\rho)^s + e^{-\alpha s}}{D_{\lambda,\rho,\alpha}T_{\lambda,\rho,\alpha}},
  \end{align*}$C_{\lambda,\rho,\alpha}=\log\rho/(\log \lambda + \alpha)$.
  Hence, the GR is
  \begin{align*}
    \bar{\mu}_1(t) &
    = \frac{(\lambda\rho)^t - e^{-\alpha t}}{T_{\lambda,\rho,\alpha}}, \quad
    \bar{\nu}_1(s) 
    = \frac{(\lambda/\rho)^s - e^{-\alpha s}}{D_{\lambda,\rho,\alpha}}, \quad
    \bar{\mu}_2(t)=e^{-\alpha t}, \\
    \bar{\nu}_2(s) &= \frac{(\lambda/\rho)^s -(\lambda\rho)^s+C_{\lambda,\rho,\alpha}(\lambda^{2s} e^{\alpha s} - e^{-\alpha s})}{D_{\lambda,\rho,\alpha}T_{\lambda,\rho,\alpha}}.
  \end{align*}

  Next, for the DT case, define $T_{\lambda,\rho,\alpha}'=1-e^{T_{\lambda,\rho,\alpha}}$ and $D_{\lambda,\rho,\alpha}'=1-e^{D_{\lambda,\rho,\alpha}}$, then
      \begin{align*}
        f_{11}^{(1)}(t) &= \sum_{b=0}^t \mu_1(b) \rho_1(b) = \sum_{b=0}^{t} (\lambda\rho)^b e^{\alpha b} = \frac{1-e^{T_{\lambda,\rho,\alpha}(t+1)}}{T_{\lambda,\rho,\alpha}'}, \\
        f_{11}^{(2)}(s) &= \sum_{a=0}^s \nu_1(a) \rho_1(a) = \sum_{a=0}^s \sbk{\frac{\lambda}{\rho}}^{a} e^{\alpha a} = \frac{1-e^{D_{\lambda,\rho,\alpha}(s+1)}}{D_{\lambda,\rho,\alpha}'},\\
        \bar{\ell}_1(s) &
        = -e^{-\alpha s} 
        \mbk{\frac{1-e^{T_{\lambda,\rho,\alpha}(t+1)}}{T_{\lambda,\rho,\alpha}'}}
        \mbk{\frac{1-e^{T_{\lambda,\rho,\alpha}(t+1)}}{T_{\lambda,\rho,\alpha}'}}
        .
      \end{align*}  
      By $(\lambda\rho)^te^{\alpha t}=e^{T_{\lambda,\rho,\alpha}t}$ and $(\lambda/\rho)^s e^{\alpha s}=e^{D_{\lambda,\rho,\alpha}s}$, 
      we have
      \begin{small}
      \begin{align*}
        & \bar{\rho}_1(s) = \sum_{b=0}^{s} \sbk{\sum_{a=0}^{s} \calK^{\DC}(b,a) u(s-a)} \rho_1(b) \\
        &=\sum_{b=0}^{s} \mbk{
          \sum_{a=0}^{b} \calK^{\DC}(b,a) u(s-a) + 
          \sum_{a=b+1}^{s} \calK^{\DC}(b,a) u(s-a)
        } e^{\alpha b} \\
        &= e^{-\alpha s} \sum_{b=0}^{s} \mbk{
          (\lambda\rho)^b\sum_{a=0}^{b} \sbk{\frac{\lambda}{\rho}}^a e^{\alpha a} + 
          \sbk{\frac{\lambda}{\rho}}^b \sum_{a=b+1}^{s} (\lambda\rho)^a e^{\alpha a}
        } e^{\alpha b} \\
        &= \frac{e^{-\alpha s}[1-2e^{T_{\lambda,\rho,\alpha}(s+1)}+e^{(T_{\lambda,\rho,\alpha}+D_{\lambda,\rho,\alpha})(s+1)}]}
        {D_{\lambda,\rho,\alpha}'T_{\lambda,\rho,\alpha}'} + 
        \frac{C_{\lambda,\rho,\alpha}'(e^{T_{\lambda,\rho,\alpha}+D_{\lambda,\rho,\alpha}}\lambda^{2s}e^{\alpha s}-e^{-\alpha s})}
        {D_{\lambda,\rho,\alpha}'T_{\lambda,\rho,\alpha}'},
      \end{align*}
      \end{small}
      where $C_{\lambda,\rho,\alpha}'
      =(e^{D_{\lambda,\rho,\alpha}}-e^{T_{\lambda,\rho,\alpha}})/(1-e^{D_{\lambda,\rho,\alpha}+T_{\lambda,\rho,\alpha}})$.
      Thus
      \begin{align*}
        \bar{\mu}_1(t) &= \frac{e^{-\alpha t}-(\lambda\rho)^te^{T_{\lambda,\rho,\alpha}}}{T_{\lambda,\rho,\alpha}'}, \quad
        \bar{\nu}_1(s) = \frac{e^{-\alpha s}-(\lambda/\rho)^s e^{D_{\lambda,\rho,\alpha}}}{D_{\lambda,\rho,\alpha}'}, \quad
        \bar{\mu}_2(t) = e^{-\alpha t}, \\
        \bar{\nu}_2(s) &= 
        \frac{e^{D_{\lambda,\rho,\alpha}}(\lambda/\rho)^s - e^{T_{\lambda,\rho,\alpha}}(\lambda\rho)^s + C_{\lambda,\rho,\alpha}'(e^{D_{\lambda,\rho,\alpha}+T_{\lambda,\rho,\alpha}}\lambda^{2s}e^{\alpha s}-e^{-\alpha s})}
        {D_{\lambda,\rho,\alpha}'T_{\lambda,\rho,\alpha}'}.
      \end{align*}
We thus show \cref{eq:output_kernel_input_exp_GR}.
\end{proof}

\subsubsection{GvR of $\Psi_{\boldeta}$}\label{sec-proof:Giv_output_GvR}
For simplicity, we drop the subscripts of such that $T=T_{\lambda,\rho,\alpha}$, $D=D_{\lambda,\rho,\alpha}$, $C=C_{\lambda,\rho,\alpha}$, $T'=T_{\lambda,\rho,\alpha}'$, $D'=D_{\lambda,\rho,\alpha}'$, and $C'=C_{\lambda,\rho,\alpha}'$.
\begin{proof}[Proof for the CT case \cref{eq:output_kernel_input_exp_GvR_CT}]
    Recall the GR \cref{eq:output_kernel_input_exp_GR} of $\Psi_{\boldeta}$ that when $k=1$, 
    \begin{align*}
        \mu_{i,1}&=\frac{(\lambda\rho)^{t_i} - e^{-\alpha t_i}}{T}, \quad \nu_{i,1}=\frac{(\lambda/\rho)^{t_i} - e^{-\alpha t_i}}{D}.
    \end{align*}
    Note that $T>0$ if and only if $\lambda\rho>e^{-\alpha}$, and $D>0$ if and only if $\lambda/\rho>e^{-\alpha}$, so we have $\mu_{i,1}\geq 0$ and $\nu_{i,1}\geq 0$ for all $i=1,\ldots,N$.
    When $i=N$, since $\hv_{N,1}$ has the same sign as $\nu_{N,1}\mu_{N,1}\geq 0$, by \cref{eq:gen2giv_vhat}, we have
    \begin{align*}
        \hv_{N,1} = \nu_{N,1}\mu_{N,1}=\frac{[(\lambda/\rho)^{t_N} - e^{-\alpha t_N}][(\lambda\rho)^{t_i} - e^{-\alpha t_i}]}{DT}\geq 0.
    \end{align*}
    When $i=N-1,\ldots,1$, by the Givens rotation \cref{eq:gen2giv_G},
    \begin{align*}
        r_{i,1} &= \sqrt{\sum_{j=i}^N \mu_{j,1}^2} = \frac{1}{\abs{T}}\sqrt{\sum_{j=i}^N [(\lambda\rho)^{t_j}-e^{-\alpha t_j}]^2}, \\
        c_{i,1} &= \frac{\mu_{i,1}}{r_{i,1}} = \frac{\abs{T}}{T}\cdot \frac{(\lambda\rho)^{t_i} - e^{-\alpha t_i}}{\sqrt{\sum_{j=i}^N [(\lambda\rho)^{t_j}-e^{-\alpha t_j}]^2}} = \frac{\abs{(\lambda\rho)^{t_i} - e^{-\alpha t_i}}}{\sqrt{\sum_{j=i}^N [(\lambda\rho)^{t_j}-e^{-\alpha t_j}]^2}}, \\
        s_{i,1} &= \frac{r_{i+1,1}}{r_{i,1}} = \frac{\sqrt{\sum_{j=i+1}^N [(\lambda\rho)^{t_j}-e^{-\alpha t_j}]^2}}{\sqrt{\sum_{j=i}^N [(\lambda\rho)^{t_j}-e^{-\alpha t_j}]^2}},
    \end{align*}
    and by $c_{i,1}\geq 0$, we have $\hv_{i,1}\geq 0$ and by \cref{eq:gen2giv_vhat},
    \begin{align*}
        \hv_{i,1} = \nu_{i,1}r_{i,1} = \abs{\nu_{i,1}}r_{i,1} = \frac{\abs{(\lambda/\rho)^{t_i} - e^{-\alpha t_i}}\sqrt{\sum_{j=i}^N [(\lambda\rho)^{t_j}-e^{-\alpha t_j}]^2}}{\abs{DT}},
    \end{align*}
    which is compatible with the case $i=N$.

    For $k=2$, 
    \begin{align*}
        \mu_{i,2}=e^{-\alpha t_i},\quad \nu_{i,2}=\frac{(\lambda/\rho)^{t_i} -(\lambda\rho)^{t_i}+C(\lambda^{2{t_i}} e^{\alpha {t_i}} - e^{-\alpha {t_i}})}{DT}.
    \end{align*}
    By $\mu_{i,2}>0$, $\hv_{N,2}$ has the same sign as $\nu_{N,2}$, and by \cref{eq:gen2giv_vhat}, we have
    \begin{align*}
        \hv_{N,2}=\nu_{N,2}\mu_{N,2}=\frac{(\lambda/\rho)^{t_N} -(\lambda\rho)^{t_N}+C(\lambda^{2{t_N}} e^{\alpha {t_N}} - e^{-\alpha {t_N}})}{DT} \cdot e^{-\alpha t_N}.
    \end{align*}
    When $i=N-1,\ldots,1$, by the Givens rotation \cref{eq:gen2giv_G},
    \begin{align*}
        r_{i,2} &= \sqrt{\sum_{j=i}^{N}\mu_{j,2}^2} = \sqrt{\sum_{j=i}^N e^{-2\alpha t_j}}, \\
        c_{i,2} &= \frac{\mu_{i,2}}{r_{i,2}} = \frac{e^{-\alpha t_i}}{\sqrt{\sum_{j=i}^N e^{-2\alpha t_j}}},
        \quad s_{i,2}=\frac{r_{i+1,2}}{r_{i,2}} = \frac{\sqrt{\sum_{j=i+1}^N e^{-2\alpha t_j}}}{\sqrt{\sum_{j=i}^N e^{-2\alpha t_j}}}.
    \end{align*}
    Since $c_{i,2}, \mu_{i,2}>0$, $\hv_{i,2}$ has the same sign as $\nu_{i,2}$, and by \cref{eq:gen2giv_vhat}, we have 
    \begin{align*}
        \hv_{i,2} = \nu_{i,2}r_{i,2} = \frac{(\lambda/\rho)^{t_i} -(\lambda\rho)^{t_i}+C(\lambda^{2{t_i}} e^{\alpha {t_i}} - e^{-\alpha {t_i}})}{DT} \cdot \sqrt{\sum_{j=i}^N e^{-2\alpha t_j}},
    \end{align*}which is compatible with the case $i=N$.
    Hence \cref{eq:output_kernel_input_exp_GvR_CT} is proved.
    \end{proof}

    \begin{proof}[Proof for the DT case \cref{eq:output_kernel_input_exp_GvR_DT}]
        When $k=1$, 
        \begin{align*}
            \mu_{i,1} = \frac{e^{-\alpha t_i}-(\lambda\rho)^{t_i}e^{T}}{T'}, \quad 
            \nu_{i,1} = \frac{e^{-\alpha t_i}-(\lambda/\rho)^{t_i} e^{D}}{D'}.
        \end{align*}
        By $e^{T}=(\lambda\rho)e^{\alpha}$ and $e^D=(\lambda/\rho)e^{\alpha}$, we have
        \begin{align*}
            T'>0 \ \Longleftrightarrow \ T<0 \ \Longleftrightarrow \ e^{-\alpha } > \lambda\rho 
            \Longleftrightarrow e^{-\alpha(t+1)} > (\lambda\rho)^{t+1} , \ t\geq 0,
        \end{align*}
        so $\mu_{i,1}\geq 0$.
        Similarly, $D'>0$ if and only if $e^{-\alpha(t+1)}>(\lambda/\rho)^{t+1}$ and so $\nu_{i,1}\geq 0$.
        When $i=N$, since $\hv_{N,1}$ has the same sign as $\mu_{N,1}\nu_{N,1}\geq 0$, by \cref{eq:gen2giv_vhat}, we have
        \begin{align*}
            \hv_{N,1} = \nu_{N,1} \mu_{N,1} = \frac{e^{-\alpha t_N}-(\lambda/\rho)^{t_N} e^{D}}{D'} \cdot \frac{e^{-\alpha t_N}-(\lambda\rho)^{t_N}e^{T}}{T'}.
        \end{align*}
        When $i=N-1,\ldots,1$, by the Givens rotation \cref{eq:gen2giv_G},
        \begin{align*}
            r_{i,1} &= \sqrt{\sum_{j=i}^N \mu_{j,1}^2} = \frac{1}{\abs{T'}} \sqrt{\sum_{j=i}^N [e^{-\alpha t_j}-(\lambda\rho)^{t_j}e^{T}]^2}, \\
            c_{i,1} &= \frac{\mu_{i,1}}{r_{i,1}} = \frac{\abs{T'}}{T'}\cdot \frac{e^{-\alpha t_i}-(\lambda\rho)^{t_i}e^{T}}{\sqrt{\sum_{j=i}^N [e^{-\alpha t_j}-(\lambda\rho)^{t_j}e^{T}]^2}} = \frac{\abs{e^{-\alpha t_i}-(\lambda\rho)^{t_i}e^{T}}}{\sqrt{\sum_{j=i}^N [e^{-\alpha t_j}-(\lambda\rho)^{t_j}e^{T}]^2}}, \\
            s_{i,1} &= \frac{r_{i+1,1}}{r_{i,1}} = \frac{\sqrt{\sum_{j=i+1}^N [e^{-\alpha t_j}-(\lambda\rho)^{t_j}e^{T}]^2}}{\sqrt{\sum_{j=i}^N [e^{-\alpha t_j}-(\lambda\rho)^{t_j}e^{T}]^2}}.
        \end{align*}
        And by $c_{i,1}\geq 0$, we have $\hv_{i,1}\geq 0$ and by \cref{eq:gen2giv_vhat},
        \begin{align*}
            \hv_{i,1} = \nu_{i,1} r_{i,1} = \abs{\nu_{i,1}} r_{i,1} = \frac{\abs{e^{-\alpha t_i}-(\lambda/\rho)^{t_i} e^{D}}\sqrt{\sum_{j=i}^N [e^{-\alpha t_j}-(\lambda\rho)^{t_j}e^{T}]^2}}{\abs{D'T'}},
        \end{align*}
        which is compatible with the case $i=N$.

        When $k=2$, $\mu_{i,2}=e^{-\alpha t_i}$ and
        \begin{align*}
            \nu_{i,2}=\frac{e^{D}(\lambda/\rho)^{t_i} - e^{T}(\lambda\rho)^{t_i} + C'(e^{D+T}\lambda^{2{t_i}}e^{\alpha {t_i}}-e^{-\alpha {t_i}})}{D'T'}.
        \end{align*}
        By $\mu_{i,2}>0$, $\hv_{N,2}$ has the same sign as $\nu_{N,2}$, and by \cref{eq:gen2giv_vhat}, we have
        \begin{align*}
            \hv_{N,2} = \nu_{N,2}\mu_{N,2} = \frac{e^{D}(\lambda/\rho)^{t_N} - e^{T}(\lambda\rho)^{t_N} + C'(e^{D+T}\lambda^{2{t_N}}e^{\alpha {t_N}}-e^{-\alpha {t_N}})}{D'T'}e^{-\alpha t_N}.
        \end{align*}
        When $i=N-1,\ldots,1$, by the Givens rotation \cref{eq:gen2giv_G},
        \begin{align*}
            r_{i,2} &= \sqrt{\sum_{j=i}^N \mu_{j,2}^2} = \sqrt{\sum_{j=i}^N e^{-2\alpha t_j}}, \\
            c_{i,2} &= \frac{\mu_{i,2}}{r_{i,2}} = \frac{e^{-\alpha t_i}}{\sqrt{\sum_{j=i}^N e^{-2\alpha t_j}}}, \quad 
            s_{i,2} = \frac{r_{i+1,2}}{r_{i,2}} = \frac{\sqrt{\sum_{j=i+1}^N e^{-2\alpha t_j}}}{\sqrt{\sum_{j=i}^N e^{-2\alpha t_j}}}.
        \end{align*}
        Since $c_{i,2},\mu_{i,2}>0$, $\hv_{i,2}$ has the same sign as $\nu_{i,2}$, and by \cref{eq:gen2giv_vhat}, we have
        \begin{align*}
            \hv_{i,2} = \nu_{i,2}r_{i,2} = \frac{e^{D}(\lambda/\rho)^{t_i} - e^{T}(\lambda\rho)^{t_i} + C'(e^{D+T}\lambda^{2{t_i}}e^{\alpha {t_i}}-e^{-\alpha {t_i}})}{D'T'} \cdot
            \sqrt{\sum_{j=i}^N e^{-2\alpha t_j}},
        \end{align*}
        which is compatible with the case $i=N$.
        Hence \cref{eq:output_kernel_input_exp_GvR_DT} is proved.
\end{proof}

\section{Additional algorithms}
\Cref{alg:Lx,alg:Ltx,alg:Lx=y,alg:Ltx=y} give the GvR-based algorithms for computing $L\x$, $L^T\x$, $L\x=\y$, and $L^T\x=\y$ for $L$ in \cref{eq:chol_L}, respectively.

\begin{algorithm}
  \caption{Triangular product $L\x$.}
  \label{alg:Lx}
  \begin{algorithmic}
  \STATE{\textbf{Input:} GvR $\c_i,\s_i,\w_i$, and $f_i$ of $L$ in \cref{eq:chol_L}, and $\y\in\R^N$.}  
  \STATE{\textbf{Output:} $\y\in\R^N$ such that $L\x=\y$} 
  \STATE{Initialize $\boldchi^\rmL\leftarrow\0_p$} 
  \FOR{$i=1\ldots,N$}
    \STATE{$y_i\leftarrow \c_i^T\boldchi^\rmL+ f_i x_i$ }
    \STATE{$\boldchi^\rmL\leftarrow \s_{i}\circ(\boldchi^\rmL+\w_{i}x_{i})$ if $i<N$}
  \ENDFOR
  \end{algorithmic}
\end{algorithm}

\begin{algorithm}
  \caption{Adjoint triangular product $L^T\x$.}
  \label{alg:Ltx}
  \begin{algorithmic}
  \STATE{\textbf{Input:} \textbf{Input:} GvR $\c_i,\s_i,\w_i$, and $f_i$ of $L$ in \cref{eq:chol_L}, and $\y\in\R^N$.}  
  \STATE{\textbf{Output:} $\y\in\R^N$ such that $L^T\x=\y$} 
  \STATE{Initialize $\boldchi^\rmR\leftarrow\0_p$} 
  \FOR{$i=N,\ldots,1$}
    \STATE{$y_i\leftarrow \w_i^T\boldchi^\rmR + f_i x_i$ }
    \STATE{$\boldchi^\rmR\leftarrow \s_{i-1}\circ (\boldchi^\rmR+\c_{i}x_{i})$ if $i>1$}
  \ENDFOR
  \end{algorithmic}
\end{algorithm}

\begin{algorithm}
  \caption{Forward substitution $L\x=\y$.}
  \label{alg:Lx=y}
  \begin{algorithmic}
  \STATE{\textbf{Input:} GvR $\c_i,\s_i,\w_i$, and $f_i$ of $L$ in \cref{eq:chol_L}, and $\y\in\R^N$.}  
  \STATE{\textbf{Output:} $\x\in\R^N$ such that $L\x=\y$.} 
  \STATE{Initialize $\boldchi\leftarrow\0_p$} 
  \FOR{$i=1,\ldots,N$}
    \STATE{$x_i\leftarrow (-c_i^T \boldchi + y_i)/f_i$}
    \STATE{$\boldchi\leftarrow \s_{i}\circ(\boldchi+\w_{i}x_{i})$ if $i<N$}
  \ENDFOR
  \end{algorithmic}
\end{algorithm}

\begin{algorithm}
  \caption{Backward substitution $L^T\x=\y$.}
  \label{alg:Ltx=y}
  \begin{algorithmic}
  \STATE{\textbf{Input:} GvR $\c_i,\s_i,\w_i$, and $f_i$ of $L$ in \cref{eq:chol_L}, and $\y\in\R^N$.}  
  \STATE{\textbf{Output:} $\x\in\R^N$ such that $L^T\x=\y$} 
  \STATE{Initialize $\boldchi\leftarrow\0_p$} 
  \FOR{$i=N,\ldots,1$}
    \STATE{$x_i=(-\w_i^T\boldchi + y_i)/f_i$ }
    \STATE{$\boldchi\leftarrow \s_{i-1}\circ (\boldchi+\c_{i}x_{i})$ if $i>1$}
  \ENDFOR
  \end{algorithmic}
\end{algorithm}

\section{Additional simulation results}\label{sec:sim_additional}
We provide additional stability tests in \Cref{sec:sim_stability_accuracy} by varying $\alpha=0.5,1.0,1.5$ in (S2) to investigate the impact of decay rate on the algorithms.
The results are shown in \Cref{fig:stab_changeAlpha}.

\begin{figure}
  \includegraphics[width=1\textwidth]{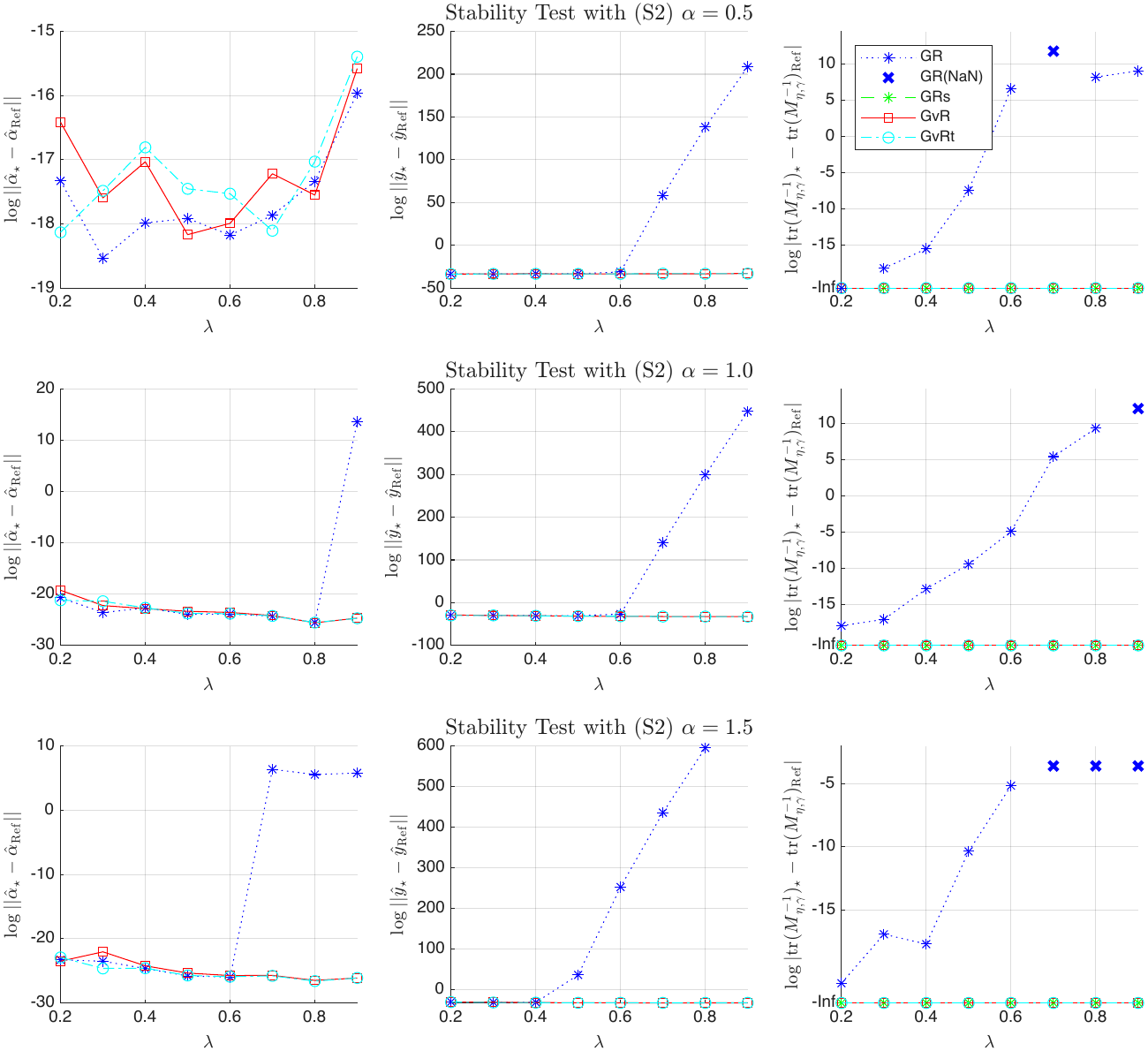}
  \caption{The logarithms of the averaged difference norms with respect to $\lambda$ using methods $\star\in\{\GR,\GRs,\GvR,\GvRt\}$ while fixing $(c,\rho,\gamma)=(1,0.6,10^{-4})$ and varying $\alpha=0.5,1.0,1.5$. 
  In the first two columns, $\GR$ and $\GRs$ are the same.  
  The experiments are repeated 80 times.}
  \label{fig:stab_changeAlpha}
\end{figure}

\bibliographystyle{siamplain}
\bibliography{references}
\end{document}